\documentclass[11pt, reqno, letterpaper]{amsart}

\usepackage[american]{babel}
\usepackage{amssymb}

\usepackage{color}

\usepackage{bbm}

\usepackage[margin = 1in]{geometry}

\usepackage{url}

\allowdisplaybreaks   

\theoremstyle{plain}
\newtheorem{de}{Definition}[section]
\newtheorem{lem}[de]{Lemma}
\newtheorem{prop}[de]{Proposition} 
\newtheorem{cor}[de]{Corollary}
\newtheorem{thm}[de]{Theorem}
\newtheorem{lemma}[de]{Lemma}

\theoremstyle{definition}
\newtheorem{rem}[de]{Remark}

\DeclareMathOperator{\curl}{curl}
\renewcommand{\div}{\operatorname{div}}
\DeclareMathOperator{\tr}{tr}

\newcommand{\RR}{{\mathbb{R}}}
\newcommand{\NN}{{\mathbb{N}}}

\newcommand{\cH}{{\mathcal{H}}}

\newcommand{\ph}{\varphi}
\newcommand{\ep}{\varepsilon}

\newcommand{\D}{\mathrm{d}}
\newcommand{\dd}{\,\mathrm{d}}

\newcommand{\sym}{\mathrm{sym}}
\newcommand{\maxi}{\mathrm{max}}
\newcommand{\tdiv}{\mathrm{div}}
\newcommand{\tcurl}{\mathrm{curl}}

\newcommand{\hra}{\hookrightarrow}

\newcommand{\beq}{\begin{equation}}
\newcommand{\eeq}{\end{equation}}

\newcommand{\ol}{\overline}
\newcommand{\wh}{\widehat}

\numberwithin{equation}{section}

\newcommand{\vertiii}[1]{{\left\vert\kern-0.25ex\left\vert\kern-0.25ex\left\vert #1 
    \right\vert\kern-0.25ex\right\vert\kern-0.25ex\right\vert}}

\begin{document}

\title{Boundary stabilization of quasilinear Maxwell equations}

\author{Michael Pokojovy}
\address{M. Pokojovy, Department of Mathematical Sciences, University of Texas at El Paso,
500 W University Ave, El Paso, TX 79968, USA}
\email{mpokojovy@utep.edu}

\author{Roland Schnaubelt}
\address{R. Schnaubelt, Department of  Mathematics,
Karlsruhe Institute of Technology, 76128 Karlsruhe, Germany}
\email{schnaubelt@kit.edu}

\thanks{We gratefully acknowledge financial support from the Deutsche Forschungsgemeinschaft (DFG) through CRC\,1173.}

\keywords{Quasilinear Maxwell equations; Silver-M\"uller boundary conditions; nonhomogeneous anisotropic materials; 
global existence; exponential stability}

\subjclass[2000]{
    35Q61,   
    35L60,   
    35L50,   
    35A02,   
    35B40,   
    35B65}    

\date\today

\begin{abstract}
We investigate an initial-boundary value problem for a quasilinear nonhomogeneous, anisotropic Maxwell system subject to an absorbing 
boundary condition of Silver \& M\"uller type in a smooth, bounded, strictly star-shaped domain of $\mathbb{R}^{3}$.
Imposing usual smallness assumptions in addition to standard regularity and compatibility conditions, 
a nonlinear stabilizability inequality is obtained by showing nonlinear dissipativity and observability-like estimates enhanced by an 
intricate regularity analysis.  With the stabilizability inequality at hand, the classic nonlinear barrier method is employed to prove that
small initial data admit unique classical solutions that exist globally and decay to zero at an exponential rate.
Our approach is based on a recently established local well-posedness theory in a class of $\cH^{3}$-valued functions.
\end{abstract}

\maketitle



\section{Introduction}
\label{SECTION:INTRODUCTION}

The goal of this work is to establish global-in-time existence and exponential stability for the quasilinear nonhomogeneous, anisotropic Maxwell 
system  in a bounded, smooth star-shaped domain of $\mathbb{R}^{3}$.
Provided the initial data are sufficiently small in the Sobolev $\cH^{3}$-norm,
we will show that the damping effect from an absorbing boundary condition of Silver \& M\"uller-type 
can render the system globally well-posed and exponentially stable.

The Maxwell equations, laying the mathematical foundation of the theory of electro-magnetism, 
establish a relation between the electric fields $E$, $D$ and the magnetic fields $B$, $H$ via Amp\`{e}re's circuital law and Faraday's law of induction. 
These basic continuity equations read as
\begin{equation*}
	\partial_{t} D = \curl H \quad \text{ and } \quad
	\partial_{t} B = -\curl H, \qquad t > 0, \ \  x \in \Omega,
	\label{EQUATION:FARADAY_AND_AMPERE_LAWS}
\end{equation*}
where $\Omega\subset \RR^3$ is a bounded domain with a smooth boundary $\Gamma$ and outer unit normal $\nu$.
In our work, we take $(E, H)$ as the state variables and postulate the instantaneous nonlinear material laws
\begin{equation*}
    \label{EQUATION:NONLINEAR_INSTANTANEOUS_MATERIAL_LAWS}
    D = \varepsilon(\cdot, E) E \quad \text{ and } \quad B = \mu(\cdot, H) H \qquad \text{ in } \Omega
\end{equation*}
with nonlinear, nonhomogeneous, anisotropic tensor-valued permittivity $\ep$ and permeability $\mu$.
Imposing a generalized linear Silver \& M\"uller boundary condition with a tensor-valued $\lambda$, we arrive at the quasilinear Maxwell system
\begin{align}
    \notag
    \partial_t \big(\ep(x, E(t, x))E(t, x)\big) &= \phantom{-}\curl H(t, x), & t &\ge 0, \ x \in \Omega, \\
    \notag
    \partial_t \big(\mu(x, H(t, x))H(t, x)\big) &= -\curl E(t, x), & t &\ge 0, \ x \in \Omega, \\
    \label{EQUATION:MAXWELL}
    \div \big(\ep(x, E) E\big) = 0, \qquad \div \big(\mu(x, H) H\big) &= 0, & t &\ge 0, \ x \in \Omega, \\
    \notag
    H(t, x) \times \nu(x) + \big(\lambda(x) \big(E(t, x) \times \nu(x)\big)\big) \times \nu(x) &= 0, & t &\ge 0, \ x \in \Gamma, \\
    \notag
    E(0, x) = E^{(0)}(x),  \qquad H(0, x) &= H^{(0)}(x), & x & \in \Omega.
\end{align}

The absorbing boundary condition in Equation \eqref{EQUATION:MAXWELL} above emerges as a first-order approximation of the more comprehensive 
``transparent'' boundary condition. But, while being dissipative in its nature, in contrast to the latter condition,
Silver \& M\"uller-type boundary conditions (sometimes \cite{PiZa1995} also -- somewhat inadequately \cite{AmNe1999} -- referred to as 
``impedance boundary conditions'') account for the waves being reflected back into the domain (cf.\ \cite[p. 136]{ElLaNi2002.2}).
From the physical point of view, such boundary conditions model the process of how electromagnetic waves are scattered by an obstacle $\Omega$
under the assumption that the underlying medium does not allow for deep penetration of the wave (cf.\ \cite[p. 20]{CaCoMo2011}).

The solenoidality condition or the divergence-freeness of $E$ and $H$ express the absence of electric and magnetic charges in the medium,
which is true in any physical environment for the  magnetic charges. This condition excludes certain non-zero equilibrium states of the system.
Because of the absence of interior conductivity or currents, the solenoidality conditions in Equation \eqref{EQUATION:MAXWELL} are uniquely related
to  the so-called ``initial solenoidality''
\begin{align}
    \label{EQUATION:INITIAL_SOLENOIDALITY}
    \div\big(\ep(\cdot, E^{(0)}) E^{(0)}\big) = 0, \qquad \div\big(\mu(\cdot, H^{(0)}) H^{(0)}\big) &= 0 \qquad \text{ in } \Omega.
\end{align}
Provided the sytem \eqref{EQUATION:MAXWELL} possesses a regular solution, the solenoidality condition in the fourth equation of 
\eqref{EQUATION:MAXWELL} trivially yields the initial solenoidality \eqref{EQUATION:INITIAL_SOLENOIDALITY}.
Vice versa, applying the $\div$-operator to the first two equations in \eqref{EQUATION:MAXWELL} and recalling the  distributional identity 
$\div \curl = 0$, the solenoidality conditions of \eqref{EQUATION:MAXWELL} follow (cf.\ \cite[Lemma~7.25]{Sp0}).
Therefore, in the following, we will be freely switching between the two equivalent conditions.

Since we are interested in classical  $\cH^{3}$-valued solutions, it is natural to expect that both regularity and compatibility conditions 
are to be imposed. Indeed, given a smooth solution on a closed time-space cylinder, an obvious requirement is that all response tensors 
(i.e., $\ep$, $\mu$ and $\lambda$), the initial data $\big(E^{(0)}, H^{(0)}\big)$ and the boundary $\Gamma$ are smooth.
Further, the higher-order time derivatives of the solution at $t = 0$ are expected to satisfy the boundary 
conditions translating into compatibility conditions on the initial data.
These have been described in the
recent well-posedness study \cite{SchSp2}. Unless the nonlinear response tensors $\ep$ and $\mu$ are globally positive definite,
the smallness of the solution and thus the initial data, is also necessary to preserve the ``hyperbolicity'' of the system \eqref{EQUATION:MAXWELL}.
Last but not least, appropriate symmetry conditions on the response tensors along with a ``tangentiality'' condition 
(cf.\ Section \ref{SECTION:SETTINGS_AND_RESULTS})
on the tensor field $\lambda$ are required.

Turning to global solutions, a uniform stabilization mechanism in the system \eqref{EQUATION:MAXWELL} needs to be discovered to pave the way 
for a reasonable long-time analysis. In contrast to semilinear problems, uniform stability is often the only way to global existence, 
especially for hyperbolic problems. In bounded domains, it is the exponential stability that furnishes the global existence of smooth solutions
originating in a vicinity of a stable equilibrium \cite{JiaRa2000, LaPoSch2018, LPW, LaPoWa2018, MePoSH2007, MRRa2002}.
In the full space $\mathbb{R}^{d}$, dispersive estimates are the path to the long-time existence 
\cite{LiPo2015,LZ, RiRa1995, RiRa1996, Ra, RaUe2017,Spe}. In case a (too) weak solution concept is employed \cite{DNS, Sp2016} 
or there is little to no dissipation (even in a smooth solution class) \cite{Al1995, Al1999.2, Al1999.1,  Al2000, Jo1981},
blow-ups are known to occur in finite time -- both for large and small data,
often independently of whether the initial data are smooth or compactly supported, etc.

In view of the dissipation present in the Maxwell system \eqref{EQUATION:MAXWELL}  at the basic energy level 
and rooted in its absorbing boundary condition,
the question arises whether this dissipative mechanism is sufficiently strong to drive the system to the zero equilibrium.
In this paper, we give a positive answer to this question in strongly star-shaped bounded $C^{5}$-domains $\Omega$ of $\mathbb{R}^{3}$ -- 
provided the initial data are small in $\cH^{3}(\Omega)$ and satisfy the compatibility conditions \eqref{ass:comp} introduced in 
Section~\ref{SECTION:SETTINGS_AND_RESULTS} to follow. Our stability proof will heavily rely on the local well-posedness theory, 
in particular, the \emph{a priori} estimates, recently established in \cite{SchSp2}
and a mixture of energy methods and control-theoretic techniques such as Rellich multipliers for proving observability/stabilizability, etc.
The core step of our reasoning  is a refined regularity analysis based on $\div$-$\curl$ ``elliptic'' theory.

We continue with a brief literature review accompanied by a general discussion. Maxwell equations  \eqref{EQUATION:MAXWELL} with symmetric and 
positive definite $\partial_E (\ep(\cdot,E)E)$ and $\partial_H (\mu(\cdot,H)H)$ can be viewed as a 
first-order quasilinear symmetric hyperbolic system. Due to their generality and ubiquity in mathematical physics, these systems have drawn 
major attention in the literature. In the full space $\Omega = \mathbb{R}^{d}$, a well-posedness theory due to Kato \cite{Ka} is known.
Applied to Equation \eqref{EQUATION:MAXWELL}, it furnishes the existence of a local-in-time $\cH^{3}$-valued solution. As for the global existence, 
for the quasilinear Maxwell system, we are only aware of very few results \cite{LZ, Ra, Spe} on the full space $\Omega = \RR^{3}$ 
that establish global existence and uniform stability of smooth solutions to small initial data. 
These studies have at their heart dispersive estimates which are not available on bounded domains. On the other hand, a blow-up behavior in 
 $\cH(\curl)$ under various boundary conditions has recently been shown in \cite{DNS}, whereas blow-up in $W^{1,\infty}(\Omega)$
has been known since a long time, cf.\ \cite{Al1995}.

Turning to domains with boundary, numerous solution theories are available \cite{Ba1979, Gu, Ka1987, Oh1981, Ra1985, Scho1986, Se1996}.
However, due to the characteristicity of the boundary symbol associated with \eqref{EQUATION:MAXWELL} (cf.\ \cite{Sp0, Sp1} for the case of 
a perfect conductor), a regularity loss in the normal direction possibly occurs making only few of these theories (e.g., \cite{Gu, Se1996}) applicable.
The additional price to pay is that a cumbersome framework of weighted Sobolev spaces of very high order needs to be employed
greatly reducing the practical applicability of these results. While the existence (without uniqueness or continuous dependence on the data) 
for Maxwell system \eqref{EQUATION:MAXWELL} has been established in the early work \cite{PiZa1995}, the full Hadamard well-posedness in  $\cH^3$ 
for perfectly conducting boundary conditions \cite{Sp0, Sp2, Sp1},
conservative interface conditions \cite{SchSp3} or absorbing boundary conditions \cite{SchSp2} as in Equation~\eqref{EQUATION:MAXWELL}
have only been proved very recently. We will strongly rely on \cite{SchSp2} in this paper.

In our main Theorem~\ref{thm:main}, we demonstrate that local $\cH^3$-valued solutions to \eqref{EQUATION:MAXWELL} can be globally extended
and decay exponentially to zero in the same topology under a smallness assumption on the initial data (in addition to natural comptability conditions).
In addition, we assume that the domain $\Omega$ is strictly star-shaped and the space derivatives of  $\ep(x,0)$ and $\mu(x,0)$ satisfy the lower bound 
\eqref{ass:coeff3}. The long-time behavior of Maxwell system has been extensively investigated in recent literature
(see, e.g., \cite{AnPo2018, El, ElLaNi2002.2, ElLaNi2002,ET12, komornik, lagnese, NiPi2003, NiPi2004, NiPi2007,phung} and references therein).
In all of these works, the authors restrict their attention to linear permittivity and permeability responses, 
often  homogeneous or isotropic ones, rendering the problem linear or semilinear.
In contrast, our problem \eqref{EQUATION:MAXWELL} is genuinely quasilinear with 
the only dissipativity source entering the system through the Silver \& M\"uller-type boundary condition.
To the best of our knowledge, apart from our companion paper \cite{LaPoSch2018} for the case of interior conductivity and perfectly conducting 
boundary conditions, no global results on \emph{quasilinear} Maxwell systems in domains with boundary are available.

In the linear situation, the aforementioned asymptotic results are furnished by an dissipativity inequality
coupled with an observability-type estimate providing a lower bound on the dissipation.
We refer the reader to \cite{komornik} for the case of scalar, constant $\ep, \mu, \lambda$;
\cite{ElLaNi2002.2,ElLaNi2002} for $x$-dependent scalar $\ep, \mu$ and nonlinear $x$-independent scalar $\lambda$;
\cite{NiPi2003} for the case of $(t, x)$-dependent scalar $\ep, \mu$ and nonlinear $x$-dependent scalar $\lambda$;
\cite{AnPo2018} for general nonohomogeneous, anisotropic $\ep, \mu$ and nonlinear $x$-independent scalar $\lambda$ and a nonlinear boundary 
feedback, etc. The use of Rellich multipliers in the present paper is motivated by the linear observability result \cite{El07}.

In our nonlinear situation, given the initial data are small, the perturbation terms also remain small till a certain time,
but the smallness is only guaranteed in much stronger norms than those controlled by the dissipation.
Hence, a rather elaborated regularity theory needs to be developed to close the regularity gap -- and, even beyond that,
with constants independent of the time interval length. Namely, for $k \in \{0, 1, 2\}$, the $\cH^{3 - k}(\Omega)$-norms of 
$\partial_{t}^{k} (E, H)$ need to be estimated by the $L^{2}$-norm of $\big(\partial^{l}_t (E, H)\big)_{l \in \{0, 1, 2, 3\}}$,
plus superlinear error terms.
This fact is established in our core Proposition~\ref{PROPOSITION:REGULARITY_BOOST} which is proved in Section~\ref{SECTION:REGULARITY_BOOST}
based on theory developed in Section~\ref{SECTION_AUXILIARY_RESULTS}. There we show an ``elliptic'' regularity result
(essentially due to \cite{CoDaNi2010}), which allows us to reconstruct the full $\cH^{1}(\Omega)$-norm of $(E, H)$ from the $L^{2}(\Omega)$-norms 
of $\curl (E, H)$ and $\div(\ep E, \mu H)$ as well as the $\cH^{1/2}(\Gamma)$-norm of the boundary data.
In contrast to our earlier work \cite{LaPoSch2018}, due to the absence of the electric resistance term,
solenoidality properties are available both for $E$ and $H$,  which facilitates the application of this fact. 
However,  higher-order normal derivatives cannot be controlled in this way since they destroy the boundary condition. 
Instead, the required normal regularity has to be deduced from the evolution and divergence equations in \eqref{EQUATION:MAXWELL},
using ideas as in \cite{LaPoSch2018}.
Appropriate localization techniques in a boundary collar also need to be employed.
Along this way, both anisotropy and inhomogeneity impose an additional challenge.
Paired with the dissipativity/observability estimates of Propositions \ref{prop:energy} \& Corollary~\ref{observe},
the regularity boost of Proposition~\ref{PROPOSITION:REGULARITY_BOOST} furnishes
the nonlinear stabilizability estimate of Proposition \ref{prop:z}. Then our main result from Theorem~\ref{thm:main} is a direct consequence
of the standard barrier method (cf.\ \cite{LaPoSch2018,LPW,  LaPoWa2018, MePoSH2007, MRRa2002}, etc.).

In the next section we discuss basic notation \& results and present our main Theorem~\ref{thm:main}.
The energy and observability-type estimates are proved in Section~\ref{sec:energy}.
In Section~\ref{SECTION:STABILIZABILITY} we state the crucial regularity result of Proposition~\ref{PROPOSITION:REGULARITY_BOOST}
before we establish Theorem~\ref{thm:main}. The proof of this proposition is based on various auxiliary results 
discussed in Section~\ref{SECTION_AUXILIARY_RESULTS}. The technically most demanding part is Section~\ref{SECTION:REGULARITY_BOOST},
where Proposition~\ref{PROPOSITION:REGULARITY_BOOST} is shown.


\section{Problem Settings and Main Results}
\label{SECTION:SETTINGS_AND_RESULTS}
Throughout this paper, let $\Omega\subset \RR^3$ be a bounded domain with a $C^{5}$-boundary $\Gamma := \partial \Omega$
and the unit outer normal vector $\nu \colon \Gamma \to \mathbb{R}^{3}$. 

For an arbitrary, but fixed $T > 0$, let $J = J_T = [0, T]$ denote the time interval, while $\Omega_T = (0, T) \times \Omega$ and $\Gamma_T = (0, T) \times \Gamma$
stand for the interior and the lateral boundary of the time-space cylinder, respectively.
For the sake of brevity and convenience, we will often make no distinction between respective spaces of scalar and vector-valued functions,
e.g., $L^{2}(\Omega)$ and $\big(L^{2}(\Omega)\big)^{3} \simeq L^{2}(\Omega, \mathbb{R}^{3})$, etc.
Further, we will frequently omit the domain of integration and merely write $\cH^k$ in lieu of $\cH^{k}(\Omega)$, etc.,
in case the domain is unambiguously clear from the context. We will always employ the standard Lebesgue measure on open domains of $\mathbb{R}^{3}$
and the associated surface measure on smooth oriented surfaces $\Gamma$ in $\mathbb{R}^{3}$.
In both cases, these measures will be denoted as $\dd x$ in respective volumetric or surface integrals.

Our central assumptions employed throughout the paper are the regularity and uniform positivity conditions of the material tensors
\begin{equation}
    \label{ass:coeff1}
    \begin{split}
        \ep, \mu &\in C^3\big(\ol{\Omega} \times \RR^3, \RR^{3 \times 3}_{\sym}\big), \quad
        \lambda \in C^{3}_{\tau}\big(\Gamma, \RR^{3 \times 3}_{\sym}\big), \\
        \ep(x, 0) &\ge 2 \eta I, \quad \mu(x, 0) \ge 2 \eta I \quad \text{ for all } x \in \ol{\Omega} \quad \text{ and } \quad
        \lambda(x) \geq \eta I \quad \text{ for all } x \in \Gamma
    \end{split}
\end{equation}
for some constant $\eta > 0$, where $I$ stands for the $(3 \times 3)$-identity matrix.
The space $C^3(\ol{\Omega})$ denotes the space of (here: tensor-valued) functions,
which -- along with their derivatives up to order three -- possess continuous extensions onto $\ol{\Omega}$.
Likewise, $C^3_\tau(\Gamma) := C^{3}(\Gamma) \cap L^{2}_{\tau}(\Gamma)$ (cf. Equation \eqref{EQUATION_TANGENTIAL_L_2})
is the space of ``tangential'' tensor-valued functions $\alpha$
such that $\operatorname{im}_{\alpha(x)}\big(\operatorname{span}\{\nu\}^{\perp}\big) \subset \operatorname{span}\{\nu\}^{\perp}$ for all $x\in\Gamma$.

Next, following \cite{LaPoSch2018}, we introduce matrix-valued ``linearizations'' $\ep^\D$ and $\mu^\D$ of $\ep$ and $\mu$ via
\begin{equation*}
    \ep^\D_{jk}(x, \xi)= \ep_{jk}(x, \xi) + \sum_{l = 1}^{3} \partial_{\xi_k} \ep_{jl}(x, \xi) \, \xi_l , \qquad
    \mu^\D_{jk}(x, \xi)= \mu_{jk}(x, \xi) + \sum_{l = 1}^{3} \partial_{\xi_k} \mu_{jl}(x, \xi) \, \xi_l
\end{equation*}
for $x\in \ol{\Omega}$, $\xi\in\RR^3$ and $j, k \in \{1, 2, 3\}$, 
which will be used to express the time-derivatives of the left-hand sides of the first two equations in \eqref{EQUATION:MAXWELL}.
We further impose the following additional regularity and symmetry conditions:
\begin{equation}
    \label{ass:coeff2}
    \partial_\xi \ep, \partial_\xi \mu \in C^3(\ol{\Omega} \times \RR^{3}, \RR^{3 \times 3}), \quad
    \ep^\D = (\ep^\D)^\top, \quad \text{ and } \quad \mu^\D = (\mu^\D)^\top,
\end{equation}
where $(\cdot)^{\top}$ denotes the standard matrix transposition operator.

The continuity of respective functions along with assumptions \eqref{ass:coeff1} furnish the existence of a (possibly) small radius 
$\tilde{\delta} \in (0, 1]$  such that the following ``local positivity'' condition
\begin{equation}
    \label{est:coeff-lower}
    \ep(x, \xi) \ge \eta I, \quad \mu(x, \xi) \ge \eta I, \quad 
    \ep^\D(x, \xi) \ge \eta I, \quad \mu^\D(x, \xi) \ge \eta I
\end{equation}
holds true for all $\xi \in\RR^3$ with $|\xi|\le \tilde{\delta}$ uniformly in $x\in \ol{\Omega}$.

\begin{rem}
    The class of functions satisfying our assumptions \eqref{ass:coeff1} and \eqref{ass:coeff2} is non-trivial.
    Indeed, apart from obvious linear candidates, physically relevant instances of genuinely nonlinear tensor responses -- both isotropic and 
    anisotropic -- can be found in \cite[Example 2.1]{LaPoSch2018}. For further examples, we refer the reader to the monograph \cite{MN}.
\end{rem}

Continuing, for a smooth vector field $u$, let $\tr_n u$ be the trace of the normal component $u \cdot \nu$ on $\Gamma$ of $u$,
while $\tr_t u$ denotes the tangential trace $u \times \nu$ of $u$ on $\Gamma$.
The tangential component 
\begin{equation}
    \nu \times (\tr_t u) = \tr u - (\tr_n u) \nu
\end{equation}
of the full trace $\tr u$ will similarly be denoted as $\tr_\tau$.
By well-known results (cf.\ Theorems~IX.1.1 and IX.1.2 in \cite{DL}), the (linear) mappings 
\begin{equation*}
    \tr_n \colon \cH(\tdiv) \to \cH^{-1/2}(\Gamma) \quad \text{ and } \quad \tr_t \colon \cH(\tcurl) \to \big(\cH^{-1/2}(\Gamma)\big)^{3}
\end{equation*}
are continuous. Here, the Hilbert spaces
\begin{align*} 
    \cH(\tcurl) &= \Big\{u \in \big(L^2(\Omega)\big)^3 \,|\, \curl u \in \big(L^2(\Omega)\big)^3\Big\}  \quad \text{ and } \quad
    \cH(\tdiv) = \Big\{u \in \big(L^2(\Omega)\big)^3 \,|\, \div u \in L^2(\Omega)\Big\}
\end{align*} 
are equipped with the natural inner products. 
Similarly, following \cite[Equation (2.12)]{NiPi2007}, we define the ``tangential'' $L^{2}$-space
\begin{equation}
    \label{EQUATION_TANGENTIAL_L_2}
    L^{2}_{\tau}(\Gamma) = \Big\{v \in \big(L^{2}(\Gamma)\big)^{3} \,\big|\, v \cdot \nu = 0 \quad \text{ on } \Gamma\Big\}
\end{equation}
endowed with the standard $L^{2}$-inner product.

In this paper, we are interested in classical solutions
\begin{equation} \label{local}
    (E, H) \in \bigcap_{k = 0}^{3} C^{k}\big([0, T_{\maxi}), \big(\cH^{3 - k}(\Omega)\big)^{6}\big)=:G^3([0, T_{\maxi}))=G^3.
\end{equation}
The space $\cH^{3}(\Omega)$ in the equation above is known to be the optimal integer-order Sobolev space 
for $E(t), H(t)$ from the quasilinear Maxwell system \eqref{EQUATION:MAXWELL}, cf.\ \cite{Sp0}.
Provided the $C^{2}$-regularity of solutions at $t = 0$, the boundary conditions in the third equation in \eqref{EQUATION:MAXWELL}
need to be satisfied by respective time-derivates of $(E, H)$ at time $t = 0$.
These facts translate into usual ``compatibility conditions'' naturally arising when treating quasilinear problems.
To express these conditions in terms of the initial data $(E^{(0)}, H^{(0)})$,
we let $E^{(0)}, H^{(0)} \in \big(\cH^3(\Omega)\big)^3$ and introduce the vector fields
\begin{align}
    \label{def:comp}
    E^{(1)} &= \big(\ep^\D(E^{(0)})\big)^{-1} \curl H^{(0)}, \qquad 
    H^{(1)}  = -\big(\mu^\D(H^{(0)})\big)^{-1} \curl E^{(0)}, \notag \\
    E^{(2)} &= \big(\ep^\D(E^{(0)})\big)^{-1} \big[\curl H^{(1)} - (\nabla_E \ep^\D(E^{(0)}) E^{(1)}) \cdot E^{(1)}\big], \\
    H^{(2)} &= -\big(\mu^\D(H^{(0)})\big)^{-1} \big[\curl E^{(1)} + (\nabla_H \mu^\D(H^{(0)}) H^{(1)}) \cdot H^{(1)}\big] \notag
\end{align}
as formal representations of $\partial_{t} \big(E(t), H(t)\big)|_{t = 0}$ for $k \in \{1, 2\}$,
where we used the convention
\begin{equation}
    \big((\nabla A) \xi \cdot \eta\big)_{j} = \Big(\sum\nolimits_{i, k} \partial_i A_{jk} \xi_{k} \eta_i\Big)_{j}.
\end{equation}
With the notation from Equations \eqref{def:comp}, for $k \in \{0, 1, 2\}$ we obtain the following compatibility and initial solenoidality 
conditions
\begin{equation} \label{ass:comp}
    \tr_t H^{(k)} +\tr_t \big(\lambda \tr_t E^{(k)}\big) = 0 \quad \text{ on } \Gamma, \qquad
        \div \big(\ep(E^{(0)}) E^{(0)}\big) =\div \big(\mu(H^{(0)}) H^{(0)}\big) = 0 \quad \text{ on } \Omega.
\end{equation}

We will proceed by invoking the local well-posedness theory recently developed in \cite{SchSp2}. Introducing the constant
\begin{equation*}
    \delta_{0} = \min\{1, \tilde{\delta}/C_S\},
\end{equation*}
where $C_{S}$ denotes the norm of the Sobolev embedding $\cH^2(\Omega)\hra L^ \infty(\Omega)$, we select numbers $T > 0$ and 
$\delta \in (0, \delta_0]$. The (small) parameter $\delta > 0$ will be chosen appropriately in subsequent proofs. 
Under the regularity and positivity conditions \eqref{ass:coeff1}--\eqref{est:coeff-lower}, the local well-posedness result in 
\cite[Theorem 6.4]{SchSp2} furnishes the existence of a (small) maximal radius $r(T, \delta) \in  \big(0, r(T, \delta_0)\big]$
such that for all radii $r \in \big(0, r(T, \delta)\big]$ and initial data $E^{(0)}, H^{(0)} \in \big(\cH^3(\Omega)\big)^3$
satisfying the compatibility conditions \eqref{ass:comp} and the smallness assumption
\begin{equation}
    \label{est:data}
    \|E^{(0)}\|_{\cH^3(\Omega)}^{2} + \|H_0\|_{\cH^3(\Omega)}^{2} \le r^{2},
\end{equation}
there exists a unique classical solution $(E, H)\in G^3([0, T_{\maxi}))$ (see \eqref{local}),
to the quasilinear Maxwell system \eqref{EQUATION:MAXWELL} up to a maximal existence time $T_{\maxi} \in (T, \infty]$.
Further, the fields $(E, H)$ have  tangential traces in $\cH^3(\Gamma_{T'})$ for $T'<T_{\maxi}$ and satisfy the estimate
\begin{equation}
    \label{est:delta}
    \max_{k \in \{0, 1, 2, 3\}} \max_{t \in [0, T]} \Big(\big\|\partial_t^k E(t)\big\|_{\cH^{3-k}(\Omega)}^2 
                                                       + \big\|\partial_t^k H(t)\big\|_{\cH^{3-k}(\Omega)}^{2}\Big) \le \delta^2 \le 1
\end{equation}
along with the solenoidality condition
\begin{equation}
    \label{eq:div0}
    \div\big(\ep(\cdot, E(t)) E(t)\big) = 0, \qquad \div\big(\mu(\cdot, H(t)) H(t)\big) = 0 \qquad \text{ in } \Omega
\end{equation}
for all $t \in [0, T_{\maxi}) =: J_{\maxi}$. Here and in the sequel, we often view vector fields $E(t, \cdot)$, etc., 
as a Hilbert space element and write $E(t)$, etc.
As noted in the introduction, condition (\ref{eq:div0}), i.e.,  the third equation in \eqref{EQUATION:MAXWELL},
easily follows from the initial solenoidality assumption in Equation \eqref{EQUATION:INITIAL_SOLENOIDALITY}
and the first two lines of \eqref{EQUATION:MAXWELL}, cf.\ \cite[Lemma~7.25]{Sp0}.

The {\it a priori} bound \eqref{est:delta} will repeatedly be invoked throughout the article either explicitly or implicitly.
This inequality will prove essential for treating the nonlinearity as it both provides a uniform bound on the solution and furnishes its smallness.
Moreover, in view of assumption \eqref{est:delta}, it guarantees the ``hyperbolicity'' of the system by preserving the uniform positive definiteness 
of respective nonlinear tensors  $\varepsilon\big(\cdot, E(t)\big)$ and $\mu\big(\cdot, H(t)\big)$ everywhere in $\Omega$ for $t \in [0, T]$,
which is indispensable both for the local existence and the global stability of regular solutions.

Fixing now $T = 1$, we obtain $r(\delta) := r(1, \delta)$.
Given arbitrary initial data fields $\big(E^{(0)}, H^{(0)}\big)$ satisfying the comptability conditions \eqref{ass:comp}
and the smallness assumption \eqref{est:data}, we introduce  the final time 
\begin{equation}
    \label{def:T*}
    T_{*} = \sup\big\{T \in [1, T_{\maxi}) \,|\, \text{global bound } \eqref{est:delta} \text{ is valid for } t \in [0, T]\big\}.
\end{equation}
Hence, the estimate \eqref{est:delta} holds true on $[0, T_{*}) =: J_{*}$.
Unless $T_{*} < \infty$, the blow-up condition in  \cite[Theorem~6.4]{SchSp2} implies $T_{\maxi} > T_{*}$ and, therefore, by continuity provided by $G^{3}$,
\begin{equation}
    \label{eq:contr}
    z(T_{*}) := \max_{k \in \{0, 1, 2, 3\}} \Big(\big\|\partial_{t}^{k} E(T_{*})\big\|_{\cH^{3-k}(\Omega)}^{2}
                                               + \big\|\partial_{t}^{k} H(T_{*})\big\|_{\cH^{3-k}(\Omega)}^{2}\Big) = \delta^{2}.
\end{equation}

To study the regularity of $(E, H)$, we need to proceed to time-differentiated versions of Equation~\eqref{EQUATION:MAXWELL}.
To this end, letting
\begin{equation}
    \label{def:ep-mu-hat}
    \wh\ep_{k} = \begin{cases} 
                    \ep(\cdot, E), & k = 0, \\  
                    \ep^\D(\cdot, E), & k \in \{1, 2, 3\},
                 \end{cases} 
                 \qquad
    \wh\mu_{k} = \begin{cases} 
                    \mu(\cdot, H), & k = 0, \\ 
                    \mu^\D(\cdot, H), & k \in \{1, 2, 3\},
                 \end{cases}
\end{equation}
 we arrive at the  non-homogeneous system
\begin{align} 
    \notag
    \partial_{t} \big(\wh\ep_{k} \partial_{t}^{k} E\big) &= \phantom{-}\curl \partial_{t}^{k} H - \partial_{t} f_{k}, & t &\in J_{\maxi}, \ x \in \Omega, \\
    \label{eq:maxwell1}
    \partial_{t} \big(\wh\mu_{k} \partial_{t}^{k} H\big) &= -\curl \partial_{t}^{k} E - \partial_{t} g_{k}, & t &\in J_{\maxi}, \ x \in \Omega, \\
    \notag
    \tr_{t} \partial_{t}^{k} H - \tr_t (\lambda \tr_{t} \partial_{t}^{k} E) &= 0, & t & \in J_{\maxi}, \ x \in \Gamma,
\end{align}
for $k \in \{0, 1, 2, 3\}$, with the commutator terms
\begin{equation}
    \label{def:fk}
    \begin{split}
        f_{0} &= f_{1} = 0, \ \ f_{2} = \big(\partial_{t} \ep^\D(\cdot, E)\big) \partial_{t} E, \ \
        f_{3} = \big(\partial_{t}^{2} \ep^\D(\cdot, E)\big) \partial_{t} E + 2 \big(\partial_{t} \ep^\D(\cdot, E)\big) \partial_{t}^{2} E, \\
        g_{0} &= g_{1} = 0, \ \ g_{2} = \big(\partial_{t} \mu^\D(\cdot, H)\big) \partial_{t} H, \ \
        g_{3} = \big(\partial_{t}^{2} \mu^\D(\cdot, H)\big) \partial_{t} H + 2 \big(\partial_{t} \mu^\D(\cdot, H)\big) \partial_{t}^{2} H.
    \end{split}
\end{equation}
In a similar fashion, differentiating Equation~\eqref{eq:div0}, the  higher-order solenoidality properties
\begin{equation}
    \label{eq:div}
    \div \big(\ep^\D(\cdot, E) \partial_{t}^{k} E\big) = -\div f_{k} \quad\text { and } \quad
    \div \big(\mu^\D(\cdot, H) \partial_{t}^{k} H\big) = -\div g_{k}  \quad \text{ for } k \in \{1, 2, 3\}
\end{equation} 
emerge.
In the latter two equations \eqref{def:fk} and \eqref{eq:div}, the functions $\partial_t f_k$, $\partial_t g_k$, $\div f_k$ and $\div g_k$ belong
to $L^{\infty}\big(J_{*}, L^{2}(\Omega)\big)$ on the strength of the estimate \eqref{est:z} stated below.
For $k = 3$, the first equation in \eqref{EQUATION:MAXWELL} is interpreted in $\cH^{-1}(\Omega_T)$,
while the divergence operator in Equation~\eqref{eq:maxwell1} is viewed as a map from $\big(L^{2}(\Omega)\big)^{3}$ into $\cH^{-1}(\Omega)$.
Since the forcing terms belong to $L^{2}(\Omega)$, the traces in Equation~\eqref{eq:maxwell1} exist as $\cH^{-1/2}(\Gamma)$-distributions
(cf.\ \cite[Section~2.1]{Sp0}).

To facilitate energy estimates, we also consider the following equivalent version of Equation~\eqref{eq:maxwell1}:
\begin{align} 
    \notag
    \ep^\D(\cdot, E) \, \partial_{t} \big(\partial_{t}^{k} E\big) &= \phantom{-}\curl \partial_{t}^{k} H - \tilde{f}_{k}, & t &\in J_{\maxi}, \ x \in \Omega, \\
    \label{eq:maxwell2}
    \mu^\D(\cdot, H) \, \partial_{t} \big(\partial_{t}^{k} H\big) &= -\curl \partial_{t}^{k} E - \tilde{g}_{k}, & t &\in J_{\maxi}, \ x \in \Omega, \\
    \notag
    \tr_{t} \partial_{t}^{k} H - \tr_t( \lambda \tr_t \partial_{t}^{k} E) &= 0, & t &\in J_{\maxi}, \ x \in \Gamma,
\end{align}
for $k\in\{0, 1, 2, 3\}$ with the new commutator terms
\begin{equation}  \label{eq:tilde-fk} \begin{split}
\tilde{f}_{0} &=0, \qquad   \tilde{f}_{k} = \sum\nolimits_{j = 1}^{k} \binom{k}{j} \big(\partial_{t}^{j} \ep^\D(\cdot, E)\big) \partial_{t}^{k + 1 - j} E 
   \qquad \text{for } k \in \{1, 2, 3\}, \\
\tilde{g}_{0} &= 0 , \qquad     \tilde{g}_{k} = \sum\nolimits_{j = 1}^{k} \binom{k}{j} \big(\partial_{t}^{j} \mu^\D(\cdot, H)\big) \partial_{t}^{k + 1 - j} H 
   \qquad \text{for } k \in \{1, 2, 3\}.
\end{split} \end{equation}

Continuing, we define  higher-order ``energies'' 
\begin{align} 
    \notag
    e_{k}(t) &= \frac{1}{2} \max_{0 \leq j \leq k} \Big(\big\|\wh\ep_{k}^{1/2} \partial_{t}^{j} E(t)\big\|_{L^2(\Omega)}^{2} 
             + \big\|\wh\mu_{k}^{1/2} \partial_{t}^{j} H(t)\big\|_{L^2(\Omega)}^{2}\Big), \qquad e := e_{3}, \\
    \label{def:dez}
    d_{k}(t) &= \max_{0 \leq j \leq k} \big\| \lambda^{1/2}\tr_{t} \partial_{t}^{j} E(t)\big\|_{L^2(\Gamma)}^{2}, \qquad d := d_{3}, \\
    \notag    z_{k}(t) &= \max_{0 \leq j \le k} \Big(\big\|\partial_{t}^{j} E(t)\big\|_{\cH^{k-j}(\Omega)}^2 
     + \big\|\partial_{t}^{j} H(t)\big\|_{\cH^{k-j}(\Omega)}^{2}\Big), \qquad   z := z_{3}
\end{align}
of order $k \in \{0, 1, 2, 3\}$ for $t \in J_{\maxi}$. The ``weight'' tensors in the definition of $e_{k}$ and $d_k$  are introduced 
in light of the energy and observability-type estimates proved in the next section.
The ``natural'' higher-order energies $e_{k}$ arise from the basic energy associated with Equation~\eqref{EQUATION:MAXWELL} 
and (\ref{eq:maxwell2}), respectively, the expression $d_{k}$ measures the squared norms of boundary damping higher-order temporal derivatives,
while $z_{k}(t)$ combine both time- and space-higher-order squared norms, which induce the topology on $G^{3}$
 after taking the maximum over $t$.
Having shown a stabilizability estimate for $e$ and $d$ with an error term involving $z$, the goal will be to 
bound $z$ through $e$ by carefully exploiting the structure of the underlying dynamics.

We proceed with useful {\it a priori} estimates on various linear and nonlinear quantities involved in our future energy estimates.
To this end, here and in the sequel, let $c_{k}$ and $c$ denote positive constants independent of 
$t \in [0, T_{*})$, $T_{*}$, $\delta \in (0,\delta_{0}]$, $r\in \big(0, r(\delta)\big]$, and the ``admissible'' initial data
$\big(E^{(0)}, H^{(0)}\big)$ satisfying the conditions \eqref{ass:comp} and \eqref{est:data}.
Adopting the arguments from \cite{SchSp2}, similar to \cite{LaPoSch2018}, we can estimate
\begin{equation} 
    \label{est:z}
    \begin{split}
        \big\|\wh\ep_{k}(t)\big\|_{\infty}, \big\|\wh\mu_{k}(t)\big\|_{\infty}, \big\|\wh\ep_{k}^{-1}(t)\big\|_{\infty}, \big\|\wh\mu_{k}^{-1}(t)\big\|_{\infty} &\le c, \\
        \big\|\partial_{(t, x)}^{\alpha} \wh\ep_{j}(t)\big\|_{L^2(\Omega)}, \big\|\partial_{(t, x)}^{\alpha} \wh\mu_{j}(t)\big\|_{L^2(\Omega)} 
        &\le c\big(z_{k}^{1/2}(t) + \delta_{\alpha_{0} = 0}\big), \\
        \max_{2 \leq k \leq 3, \, 0 \leq j \leq 1} \Big(\big\|\partial_{t}^{j} f_{k}(t)\big\|_{\cH^{4-j-k}(\Omega)} 
        + \big\|\partial_{t}^{j} g_{k}(t)\big\|_{\cH^{4-j-k}(\Omega)}\Big) &\le cz(t), \\
        \big\|f_{2}(t)\big\|_{L^{2}(\Omega)}, \big\|g_{2}(t)\big\|_{L^{2}(\Omega)}, \big\|f_{3}(t)\big\|_{L^{2}(\Omega)}, \big\|g_{3}(t)\big\|_{L^{2}(\Omega)}
        &\le c e_{2}^{1/2}(t), \\
        \big\|\tilde f_{k}(t)\big\|_{\cH^{3-k}(\Omega)}, \big\|\tilde g_{k}(t)\big\|_{\cH^{3-k}(\Omega)} &\le c z(t)
    \end{split} 
\end{equation} 
for all $0 \leq j, k \leq 3$, $\alpha \in \NN_{0} \times \NN_{0}^{3}$ with $|\alpha| = k > 0$ and $t \in J_{*}$,
where, as stated before, the constants $c$ do not depend on $t$. 
The expression $\delta_{\alpha_{0} = 0}$ is the Kronecker delta defined as $\delta_{\alpha_{0} = 0} = 1$ if $\alpha_{0} = 0$ 
and $\delta_{\alpha_{0} = 0} = 0$ if $\alpha_{0} > 0$. 
The ``corrrection'' term ``$+c$'' on the right-hand side of the second line of Equation \eqref{est:z}
emerges from applying the $\partial^{\alpha}$-operator to the $x$-variable of $\ep(\cdot, E)$ or $\mu(\cdot, H)$.

As previously announced, the thrust of this paper is to prove that the Maxwell system \eqref{EQUATION:MAXWELL} subject to a Silver \& M\"uller-type boundary damping
admits a unique classical solution in the solution space $G^{3}$ defined in Equation \eqref{local} provided the initial data are sufficiently small
and satisfy appropriate regularity and compatibility conditions.
When studying a quasilinear dynamics in bounded domains, it is commonly recognized 
that the global well-posedness goes hand in hand with the exponential stability of classical solutions
(cf.\ \cite{LaPoSch2018, LPW, LaPoWa2018, MePoSH2007}, etc.).
Hence, a big part of the present paper is devoted to showing the desired stability property.

For our main result  we need two more assumptions. First, $\Omega$ has to be 
strictly star-shaped with respect to some point $x_{0} \in \mathbb{R}^{3}$, i.e., there exists a number $\ol{\eta} > 0$ such that
    \begin{equation}
        \label{EQUATION:STRICT_STAR_SHAPEDNESS}
        \nu \cdot m \geq \ol{\eta} > 0 \quad \text{on }  \Gamma \qquad \text{with \ }  m(x) = x - x_{0}.
    \end{equation}
Second, we suppose that the permittivity and permeability tensors satisfy the lower bounds
\begin{equation}\label{ass:coeff3}
   \ep(x,0) + (m(x)\cdot\nabla_x)\ep(x,0) \ge  \kappa\ep(x,0) \quad \text{ and } \quad   
   \mu(x,0) + (m(x)\cdot\nabla_x)\mu(x,0) \ge  \kappa\mu(x,0)
\end{equation}
for a constant $\kappa>0$ and all $x\in\ol{\Omega}$. Variants of this condition have  often been used in the nonhmogeneous linear and semilinear cases 
(see, e.g., \cite{AnPo2018, El07,NiPi2003, NiPi2004}). Of course, it is satisfied in the homogeneous case where the tensors do not depend on $x$.

We state our main result which is proved at the end of Section~\ref{SECTION:STABILIZABILITY}.
\begin{thm}
    \label{thm:main}
    Let $\Omega \subset \RR^{3}$ be a bounded, simply connected domain with a smooth boundary $\Gamma := \partial \Omega \in C^{5}$
    satisfying the strict star-shapedness condition \eqref{EQUATION:STRICT_STAR_SHAPEDNESS}. 
    Further, let the coefficient tensors $\ep$, $\mu$ and $\lambda$ satisfy the assumptions \eqref{ass:coeff1}, \eqref{ass:coeff2} and \eqref{ass:coeff3},
    and let the initial data $E^{(0)}, H^{(0)} \in \big(\cH^{3}(\Omega)\big)^{3}$
    fulfill the compatibility conditions \eqref{ass:comp} and the smallness assumption \eqref{est:data}. 
    Then there exists a radius $r > 0$ entering in assumption \eqref{est:data} and stability constants  $M, \omega > 0$ 
    such that the classical solution $(E, H) \in G^{3}$ of the Maxwell system \eqref{EQUATION:MAXWELL} uniquely exists for all times $t \ge 0$ 
    and exhibits an exponential decay rate
    \begin{equation*}
        \max_{0 \le j \le 3} \Big(\big\|\partial_{t}^{j} E(t)\big\|_{\cH^{3-j}(\Omega)}^{2} + \big\|\partial_{t}^{j} H(t)\big\|_{\cH^{3-j}(\Omega)}^{2}\Big) 
        \le M e^{-\omega t} \big\|\big(E^{(0)}, H^{(0)}\big)\big\|_{\cH^3(\Omega)}^{2} \quad \text{ for all } t \ge 0.
    \end{equation*}
\end{thm}


\section{Energy and Observability-Type Estimates}\label{sec:energy}

We first establish the basic dissipation of the system from the energy loss at the boundary through the Silver \& M\"uller-type 
boundary condition. This property will be expressed via the quantities $e_{k}$ and $d_{k}$ at all energy levels $0 \leq k \leq 3$.
In the proof we apply standard energy techniques.
\begin{prop}[Dissipativity inequality]
    \label{prop:energy}
Let the assumptions of Theorem \ref{thm:main} -- with the exception of the simple connectedness of $\Omega$ and \eqref{ass:coeff3} -- be satisfied.
    Then, for $0 \leq k \leq 3$ and $0 \leq s \leq t \leq T_{\ast}$, the energy estimate
    \begin{equation}
        \label{est:energy}
        e_{k}(t) + \int_{s}^{t} d_{k}(\tau) \dd \tau \leq e_{k}(s) + c_{1} \int_{s}^{t} z^{3/2}(\tau) \dd \tau
    \end{equation}
    holds true with positive constants $c_{1}$ and $\eta$ independent of $s, t$.
\end{prop}

To streamline the exposition, we consider the following linear non-autonomous Maxwell system (motivated by \eqref{eq:maxwell2}).
Let $\alpha, \beta \in W^{1, \infty}\big(J, L^{\infty}(\Omega, \mathbb{R}^{3 \times 3}_{\mathrm{sym}})\big)$ with
$\alpha, \beta \geq \eta$, $\varphi, \psi \in L^{2}(\Omega_{T})^3$, $u^{(0)}, v^{(0)} \in \big(L^{2}(\Omega)\big)^{3}$.
By \cite[Proposition 3.1]{SchSp2}, the non-homogeneous problem
\begin{align}
    \notag
    \alpha \partial_{t} u&= \phantom{-}\curl v + \varphi & &\text{in } J \times \Omega, \\
    \notag
   \beta \partial_{t} v&= -\curl u + \psi & &\text{in } J \times \Omega, \\
    \label{EQUATION:LINEAR_INHOMOGENEOUS_MAXWELL_SYSTEM}
    v \times \nu + (\lambda (u \times \nu)) \times \nu &= 0 & &\text{in } J \times \Gamma, \\
    \notag
    u(0) = u^{(0)}, \quad v(0) &= v^{(0)} & &\text{in } \Omega
\end{align}
possesses a unique weak solution $(u, v) \in C^{0}\big([0, T], (L^{2}(\Omega))^6\big)$
with $(\tr_{\tau} u, \tr_{\tau} v) \in L^{2}(\Omega_{T}, \mathbb{R}^{6})$. See also the earlier paper \cite{CaEl2011} for a slight variant.
The latter extra regularity of the trace in time and space 
-- in addition to what one would expect from the mapping $\tr_{\tau} \colon L^{2}(\Omega) \to \cH^{-1/2}(\Gamma)$ --
is referred as ``hidden regularity'' and often occurs in hyperbolic problems such as the wave equation \cite{LaLiTri1986, LaTri1990, LaTri1991}, 
the dynamic equations of elasticity \cite{KuMaTu2013}, fluid/structure interaction systems \cite{ChuLa2012}, plate and beam equations \cite{LaMaMcDe2012, LaTou2010}, etc.  
Proposition~3.1 in \cite{SchSp2} also yields the following energy equality.
\begin{lemma}
    \label{LEMMA:BASIC_ENERGY_ESTIMATE}
    Under the above assumptions, for $0 \leq s \leq t \leq T$ we have
    \begin{align*}
        \big\|\alpha^{1/2}(T) u(T)\big\|_{L^{2}(\Omega)}^{2} &+ \big\|\beta^{1/2}(T) v(T)\big\|_{L^{2}(\Omega)}^{2} +
        2 \int_{s}^{t} \|\lambda \tr_t u \cdot \tr_t u\|_{L^{2}(\Gamma)}^{2} \dd \tau \\
        &= \big\|\alpha^{1/2}(0) u(0)\big\|_{L^{2}(\Omega)}^{2} + \big\|\beta^{1/2}(0) v(0)\big\|_{L^{2}(\Omega)}^{2} +
        2 \int_{s}^{t} \int_{\Omega} (u \cdot \varphi + v \cdot \psi) \dd x \dd \tau \\
        &\qquad+ \int_{s}^{t} \int_{\Omega} \big(u \cdot (\partial_{t} \alpha) u + v\cdot (\partial_{t} \beta) v\big) \dd x \dd \tau. 
    \end{align*}
\end{lemma}

In view of \eqref{eq:maxwell2}, Proposition \ref{prop:energy} for $k\in\{1,2,3\}$ is a consequence of the bounds in \eqref{est:z}
and Lemma~\ref{LEMMA:BASIC_ENERGY_ESTIMATE} with $\alpha=\wh\ep_k$, $\beta=\wh\mu_k$, $\ph=-\tilde{f}_k$, and $\psi=-\tilde{g}_k$.
To show Proposition \ref{prop:energy} for $k=0$, we can use the ample regularity of our solution $(E,H)\in G^3$ to Maxwell system \eqref{EQUATION:MAXWELL}.
Using  the symmetry of the coefficients, inserting Equation \eqref{EQUATION:MAXWELL} and integrating by parts, we compute
\begin{align*}
    \partial_{t} &\int_{\Omega} \big(\wh\ep_0 E \cdot E + \wh\mu_0 H \cdot H\big) \dd x \\
    &= 2 \int_{\Omega} \Big(\partial_{t} \big(\wh\ep_0 E \big) \cdot E + \partial_{t} \big(\wh\mu_0 H\big) \cdot H\Big) \dd x 
    + \int_{\Omega} \Big(\wh\ep_0 E \cdot \big(\partial_{t} \wh\ep_0^{-1}\big) \wh\ep_0 E + \wh\mu_0 H \cdot \big(\partial_{t} \wh\mu_0^{-1}\big) \wh\mu_0 H\Big) \dd x \\
    &= 2 \int_{\Omega} (\curl H \cdot E - \curl E \cdot H) \dd x
    - \int_{\Omega} \Big(E \cdot \big(\partial_{t} \wh\ep_0\big) E + H \cdot \big(\partial_{t} \wh\mu_0\big) H\Big) \dd x \\
    &= -2 \int_{\Gamma} (H \times \nu) \cdot E \dd x 
    - \int_{\Omega} \Big(E \cdot \big(\partial_{t} \wh\ep_0\big) E + H \cdot \big(\partial_{t} \wh\mu_0\big) H\Big) \dd x \\
    &= 2 \int_{\Gamma} \big((\lambda(E \times \nu)) \times \nu\big) \cdot E \dd x
    - \int_{\Omega} \Big(E \cdot \big(\partial_{t} \wh\ep_0\big) E  + H \cdot \big(\partial_{t} \wh\mu_0\big) H\Big) \dd x \\
    &= -2 \int_{\Gamma} |\lambda ^{1/2} (E \times \nu)|^{2} \dd x
    -\int_{\Omega} \Big(E \cdot \big(\partial_{t} \wh\ep_0\big) E + H \cdot \big(\partial_{t} \wh\mu_0\big) H\Big) \dd x.
\end{align*}
We now integrate in time over $(s, t)$ with respect to $\tau$ and apply the estimates in \eqref{est:z}. 
Proposition~\ref{prop:energy} also follows for $k=0$, too.

Next, we aim to prove a nonlinear observability-type estimate involving the higher-order natural energies $e_{k}$ and the dissipation 
functionals $d_{k}$. To this end, we now have to assume that $\Omega$ is strictly star-shaped 
(viz.\ Equation \eqref{EQUATION:STRICT_STAR_SHAPEDNESS}) and that $\ep$ and $\mu$ satisfy the lower bound \eqref{ass:coeff3}. 

\begin{prop}[Observability-like inequality]
    \label{prop:lower}
    Under the assumptions of Theorem~\ref{thm:main}, there is a radius $\delta_1\in (0,\delta_0]$ such that for all $\delta\in(0,\delta_1]$
in    \eqref{est:delta}, the estimate
    \begin{equation*}
        \int_{s}^{t} e_k(\tau) \dd \tau \leq c_{2} \int_{s}^{t} d_{k}(\tau) \dd \tau + c_3 \big(e_{k}(t) + e_{k}(s)\big) + c_{4} \int_{s}^{t} z^{3/2}(\tau) \dd \tau
    \end{equation*}
    holds true for $k \in \{0, 1, 2, 3\}$, $0 \le s \le t < T_{*}$, and some constants $c_{j}$ independent of $s$ and $t$.
\end{prop}

To prove Proposition \ref{prop:lower} for $k\ge1$, we again look at a non-autonomous Maxwell system. Since now the divergence conditions are involved, 
it is more convenient to use the system
\begin{align}
    \notag
    \partial_{t}(\alpha  u)&= \phantom{-}\curl v + \partial_t\varphi & &\text{in } J \times \Omega, \\
    \notag
    \partial_{t}(\beta  v)&= -\curl u + \partial_t \psi & &\text{in } J \times \Omega, \\
    \label{EQUATION:LINEAR_INHOMOGENEOUS_MAXWELL_SYSTEM1}
    v \times \nu + (\lambda (u \times \nu)) \times \nu &= 0 & &\text{in } J \times \Gamma, \\
    \notag
    u(0) = u^{(0)}, \quad v(0) &= v^{(0)} & &\text{in } \Omega,
\end{align}
which better matches \eqref{eq:maxwell1}.
We take coefficients $\alpha, \beta \in C^1\big(J\times\overline{\Omega}, \mathbb{R}^{3 \times 3}_{\mathrm{sym}})\big)$ with
$\alpha, \beta \geq \eta$ and data $\varphi, \psi \in C^1([0,T], L^2(\Omega)^3)\cap  C([0,T], \cH^1(\Omega)^3)=G^1$
and $u^{(0)}, v^{(0)} \in \big(L^{2}(\Omega)\big)^3$ such that  $\div\ph, \div \psi\in C(J,L^2(\Omega))$,
$\div (\alpha(0) u^{(0)})= \div \ph(0)$, and  $\div(\beta(0) v^{(0)})=\div \psi(0)$. 
Proposition~1.1 of \cite{CaEl2011} or Proposition 3.1 of \cite{SchSp2}  provide  a unique weak solution $(u, v) \in C^{0}\big([0, T], (L^{2}(\Omega))^6\big)$
of \eqref{EQUATION:LINEAR_INHOMOGENEOUS_MAXWELL_SYSTEM1} with $(\tr_{\tau} u, \tr_{\tau} v) \in L^{2}(\Omega_{T}, \mathbb{R}^{6})$. We will see below that
our assumptions and Theorem~1.2 of \cite{CaEl2011} imply that the full traces  $(\tr u, \tr v)$ belong to $L^{2}(\Omega_{T}, \mathbb{R}^{6})$.
We further assume that
\begin{align}\notag
    &\|\alpha\|_{L^{\infty}}, \|\beta\|_{L^{\infty}}, \|\lambda\|_{L^{\infty}} \leq M,\qquad \text{ for some } M> 0,\\
    \label{EQUATION:DEFINITION_ALPHA_TILDE_BETA_TILDE}
    &\tilde{\alpha} := \alpha + (m \cdot \nabla) \alpha \geq \tfrac{\kappa\alpha}{2} \geq \tfrac{\kappa\eta}{2}\,, \quad \text{ and } \quad
    \tilde{\beta}  := \beta  + (m \cdot \nabla) \beta  \geq \tfrac{\kappa\beta}{2}  \geq \tfrac{\kappa\eta}{2}
\end{align}
for some constants $\eta,\kappa>0$,  which are given by \eqref{ass:coeff2} and \eqref{ass:coeff3} later on.

\begin{lemma}\label{LEMMA:LINEAR_BOUNDARY_OBSERVABILITY}
Besides the assumptions stated above, we also require that the domain $\Omega$ is strictly star-shaped 
(viz.\ Equation~\eqref{EQUATION:STRICT_STAR_SHAPEDNESS}). Let  $(u, v) \in C^{0}\big([0, T], L^{2}(\Omega, \RR^{6})\big)$
be the weak solution of Equation \eqref{EQUATION:LINEAR_INHOMOGENEOUS_MAXWELL_SYSTEM1}. 
Then the traces  $(\tr u, \tr v)$ belong to $L^{2}(\Omega_{T}, \mathbb{R}^{6})$ and the fields $(u,v)$ fulfill the estimate
\begin{align}\label{est:obs}
  \frac{\kappa}{4}& \int_{s}^{t} \int_{\Omega} \big(\alpha u \cdot u + \beta v \cdot v\big) \dd x \dd \tau
     + \frac{\ol{\eta}}{4} \int_{s}^{t} \int_{\Gamma} \tr\big(\alpha u \cdot u + \beta v \cdot v\big) \dd x \dd \tau \\
  &\leq \frac{\|m\|_{L^{\infty}}^{2}}{\ol{\eta}} M (1 + M)^{2} \int_{s}^{t} \int_{\Gamma} |u \times \nu|^{2} \dd x \dd \tau 
   +\frac{M\|m\|_{L^{\infty}}}{2}\big(\|(u(t),v(t))\|_{L^{2}}^{2}  + \|(u(s),v(s))\|_{L^{2}}^{2} \big) \notag\\
& \qquad+ \|m\|_{L^{\infty}} \int_{s}^{t} \int_{\Omega} \big(M (|\partial_t\varphi| \cdot |v| + |\partial_t\psi| \cdot |v|)
         +\big(|\div\ph| \, |u| + |\div\psi| \, |v|\big) \dd x \dd \tau\notag
    \end{align}
 for $0 \leq s \leq t \leq T$.
\end{lemma}

\begin{proof} 
 We let $s = 0$ without loss of generality.

 1) We  start by regularizing the data. To this end, we first look at the homogeneous autonomous version of \eqref{EQUATION:LINEAR_INHOMOGENEOUS_MAXWELL_SYSTEM1}
 with $\ph=\psi=0$, time-independent coefficients $\alpha(0)$ and $\beta(0)$, and initial ``charges'' $\div (\alpha(0) u^{(0)}),\div(\beta(0) v^{(0)})\in L^2(\Omega)$. 
 In this case  the coefficients commute with the time derivative.
 Theorem~1.2 of \cite{CaEl2011} implies that the weak solutions $(\tilde u,\tilde v)$ of 
  \eqref{EQUATION:LINEAR_INHOMOGENEOUS_MAXWELL_SYSTEM1} with these coefficients and data form a $C_0$--semigroup on the space
  \[ X=\{ (u_0,v_0)\in L^2(\Omega)^6\,|\, \div (\alpha(0) u_0),\div(\beta(0)v_0)\in L^2(\Omega)\}\]
 endowed with the norm given by $\|u_0\|_2^2+\|v_0\|_2^2 +  \|\div (\alpha(0) u_0)\|_2^2+ \|\div(\beta(0)v_0)\|_2^2$. (We note that 
 $ \div (\alpha(0) \tilde u(t))= \div (\alpha(0) u_0)$ and $\div(\beta(0)\tilde v(t)) =\div (\alpha(0) u_0)$ for all $t\ge0$. Moreover, one has to use the remarks
  at the end of the first section  of \cite{CaEl2011}, to extend the results of this paper from scalar $\lambda$ to positive definite and bounded ones.)
  This semigroup has a generator $A$ whose domain $D(A)$ is dense in $X$. For data $(u^{(0)}, v^{(0)})\in D(A)$ the solution $(\tilde u,\tilde v)$ is an element of
  $C^1([0,\infty),X)$, and hence $(\tilde u(t),\tilde  v(t))$  is an element of 
  $\cH(\curl)^2$ by the evolution equations. Since these fields  also satisfy the boundary condition in 
 \eqref{EQUATION:LINEAR_INHOMOGENEOUS_MAXWELL_SYSTEM1},   Lemma~\ref{prop:div-curl} shows that $(\tilde u,\tilde v)$ is contained in $C([0,\infty),\cH^1(\Omega)).$ 
 (Of course, the proof of this lemma is independent of the present one.) In particular, the initial values  $(\tilde u(0),\tilde v(0))\in (A)$ 
 are contained in $\cH^1(\Omega)$ and fulfill the  boundary condition.
 
 Let now $(u^{(0)}, v^{(0)})$, $\ph$, and $\psi$ be given as in the statement, and let $(u,v)\in C\big([0, T], L^{2}(\Omega)\big)$ be the solution of 
 \eqref{EQUATION:LINEAR_INHOMOGENEOUS_MAXWELL_SYSTEM1}.  In view of the previous paragraph, there are 
functions $(u_n^{(0)}, v_n^{(0)})$ in $\cH^1(\Omega)$ that satisfy the third line of \eqref{EQUATION:LINEAR_INHOMOGENEOUS_MAXWELL_SYSTEM1}
and tend to $(u^{(0)}, v^{(0)})$ in $X$.  We can also construct maps $\ph_n,\psi_n\in C^2(\ol{\Omega_T})$ which  tend to 
 $\ph$ respectively $\psi$ in $G^1$.  Theorem~1.3 in \cite{CaEl2011} then provides solutions $(u_n, v_n) \in G^1$  of 
 \eqref{EQUATION:LINEAR_INHOMOGENEOUS_MAXWELL_SYSTEM1}. 
 Theorem~1.2 in \cite{CaEl2011} now shows that $(u_n, v_n)$ and  $(\tr u_n, \tr v_n)$ tend to 
 $(u, v)$ and  $(\tr u, \tr v)$ in $C\big([0, T], L^{2}(\Omega)\big)$ respectively $L^2(\Gamma_T)$. Below we establish the asserted estimate 
for $(u_n,v_n)$, so that the claims will then follow by approximation.

 2) We may thus assume that $(u, v)$ belongs to $G^1$ and use the given data. 
 We set $\Omega_{t} := (0, t) \times \Omega$ and $\Gamma_{t} := (0, t) \times \Gamma$.
    Now, in view of Equation \eqref{EQUATION:LINEAR_INHOMOGENEOUS_MAXWELL_SYSTEM1}, we can write
\begin{align}      
J &:= \int_{\Omega} \partial_{t} \big((m \times \alpha u) \cdot \beta v\big) \dd x 
        \label{EQUATION:OBSERVABILITY_ESTIMATE_FUNCTIONAL_I_DIFFERENTIATED}\\
&= \int_{\Omega} \Big(\big(m \times (\curl v + \partial_t \varphi)\big) \cdot \beta v 
  + (m \times \alpha u) \cdot (-\curl u + \partial_t\psi)\Big) \dd x \notag\\
 &= -\int_{\Omega} \curl v \cdot (m \times \beta v) \dd x - \int_{\Omega} \curl u \cdot (m \times \alpha u) \dd x 
    +\int_{\Omega} \big(m \times \partial_t\varphi) \cdot \beta v + (m \times \alpha u) \cdot \partial_t\psi\big) \dd x.\notag
    \end{align}
    Next, we proceed as in \cite[Equations (3.6)--(3.8)]{El07}. First, we obtain
    \begin{align*}
        \int_{\Omega} \curl v \cdot (m \times \beta v) \dd x
        = \int_{\Omega} v \cdot \curl(m \times \beta v) \dd x + \int_{\Gamma} (\nu \times v) \cdot (m \times \beta v) \dd x
    \end{align*}
and compute
    \begin{align*}
        \curl(m \times \beta v) &= \div (\beta v) m + (\beta v \cdot \nabla) m - \div(m) \beta v - (m \cdot \nabla) (\beta v) \\
        &= \div(\beta v) m - 2\beta v - (m \cdot \nabla) (\beta v).
    \end{align*}
    Applying $\div$ the second equation in (\ref{EQUATION:LINEAR_INHOMOGENEOUS_MAXWELL_SYSTEM}) and integrating in time, it follows
    \[ \div(\beta(t) v(t)) = \div(\beta(0) v^{(0)}) + \div \psi(t) - \div\psi(0)=  \div \psi(t) .\]
     Integrating over $(0, t)$ and perfoming partial integration, we arrive at
    \begin{align*}
        \int_{\Omega_{t}} \curl v \cdot (m \times \beta v) \dd(\tau, x) &=
        -\frac{1}{2} \int_{\Omega_{t}} |\tilde{\beta}^{1/2} v|^{2}  \dd(\tau, x) 
         - \frac{1}{2} \int_{\Gamma_{t}} (\nu \cdot m) (\beta v \cdot v) \dd(\tau, x) \\
        &\qquad + \int_{\Omega_{t}} m \cdot v \div\psi \dd(\tau, x) + \int_{\Gamma_{t}} (\nu \times v) \cdot (m \times \beta v) \dd(\tau, x).
    \end{align*}
    (Recall the definition of $\tilde{\alpha}$ and $\tilde{\beta}$ in Equation \eqref{EQUATION:DEFINITION_ALPHA_TILDE_BETA_TILDE}.)
    We proceed analogously with $u$. Invoking the identity (\ref{EQUATION:OBSERVABILITY_ESTIMATE_FUNCTIONAL_I_DIFFERENTIATED})
    and using the boundary condition in (\ref{EQUATION:LINEAR_INHOMOGENEOUS_MAXWELL_SYSTEM1}), we get
    \begin{align*}
\int_{\Omega} &\big(m \times \alpha(t) u(t)\big) \beta(t) v(t) \dd x - \int_{\Omega} \big(m \times \alpha(0) u(0)\big) \beta(0) v(0) \dd x
        = \int_{0}^{t} J \dd \tau \\
&= \frac{1}{2} \int_{\Omega_{t}} \big(\tilde{\alpha} u\cdot u + \tilde{\beta} v \cdot v\big) \dd(\tau, x)
  + \frac{1}{2} \int_{\Gamma_{t}} (\nu \cdot m) \tr(\alpha u \cdot u + \beta v \cdot v) \dd(\tau, x) \\
&\quad- \int_{\Gamma_{t}} \big((\nu \times u) \cdot (m \times \alpha u) + (\lambda( u\times \nu) \times \nu) \cdot(m \times \beta v)\big) \dd(\tau, x) \\
&\quad+ \int_{\Omega_{t}} \Big((m \times \partial_t\varphi) \cdot \beta v 
   + (m \times \alpha u) \cdot \partial_t\psi - \big((m \cdot u \div\ph + m \cdot v \div\psi\Big) \dd(\tau, x).
    \end{align*}
By virtue of the strict star-shapedness condition \eqref{EQUATION:STRICT_STAR_SHAPEDNESS}
    and the assumption \eqref{EQUATION:DEFINITION_ALPHA_TILDE_BETA_TILDE}, we now obtain
    \begin{align*}
        \frac{\kappa}{4} \int_{\Omega_{t}}& \big(\alpha u \cdot u + \beta v \cdot v) \dd(\tau, x)
        + \frac{\ol{\eta}}{2} \int_{\Gamma_{t}} \tr(\alpha u \cdot u + \beta v \cdot v) \dd(\tau, x) \\
        &\leq \frac{\|m\|_{L^{\infty}}}{2} M^{2} \big(\|u(t)\|_{L^{2}}^{2} + \|v(t)\|_{L^{2}}^{2} + \|u(0)\|_{L^{2}}^{2} + \|v(0)\|_{L^{2}}^{2}\big) \\
        &\quad + \frac{M\|m\|_{L^{\infty}}^2}{4} (1 + M)^{2} \frac{1}{\theta} \int_{\Gamma_{t}} |u \times \nu|^{2} \dd(\tau, x)
        + \theta \int_{\Gamma_{t}} \tr(\alpha u \cdot u + \beta v \cdot v) \dd(\tau, x) \\
        &\quad+ \|m\|_{L^{\infty}} \int_{\Omega_{t}} \Big[M\big(|\partial_t\varphi| \,|v| + |\partial_t\psi| \, |u|\big) +
        \big(|\div\ph| \, |u| + |\div\psi| \, |v|\big)\Big] \D(\tau, x),
    \end{align*}
    where we applied Young's inequality. 
    Selecting $\theta = \ol{\eta}/4$, the assertion follows.
\end{proof}
 
We are now able to prove Proposition~\ref{prop:lower}.  First, let $k\in\{1,2,3\}$. 
 By continuity and \eqref{ass:coeff3}, we can choose a number $\delta_1 \in (0, \delta_{0}]$ such that  the tensors $\alpha :=\wh\ep_k=\ep^\D(\cdot,E)$ 
 and $\beta := \wh\mu_k=\mu^\D(\cdot,H)$ 
 satisfy the conditions \eqref{EQUATION:DEFINITION_ALPHA_TILDE_BETA_TILDE} for all $\delta \in (0, \delta_1]$ in \eqref{est:delta}.
Further, let $\varphi := -f_{k}$ and $\psi := - g_{k}$ for the functions $f_{k}$ and $g_{k}$ defined in Equation~\eqref{def:fk}. 
As in \eqref{est:z}, one can see that they belong to $G^1$ since $(E,H)\in G^3$.
Equation \eqref{eq:div} at $t=0$  yields $\div (\alpha(0) u^{(0)})= \div \ph(0)$ and  $\div(\beta(0) v^{(0)}) =\div \psi(0)$.
So we can apply  Lemma~\ref{LEMMA:LINEAR_BOUNDARY_OBSERVABILITY}. Combined with Equations \eqref{eq:div} and \eqref{est:z}, it implies
the assertion of  Proposition~\ref{prop:lower} for $k\ge 1$. The case $k=0$ can be treated in the same way with $\ph=\psi=0$, but here one 
does not need a regularization step since $(E,H)\in G^3$.

Arguing as in the proof of \cite[Corollary 3.5]{LaPoSch2018}, Proposition~\ref{prop:energy} and \ref{prop:lower}
furnish the following ``incomplete'' nonlinear stabilizability inequality for time-like higher-order energies $e_{k}$.
\begin{cor}
    \label{observe}
    Suppose the conditions of Theorem \ref{thm:main} hold true and that $\delta\in(0,\delta_1]$ in \eqref{est:delta} with the number
    $\delta_1$ from Proposition~\ref{prop:lower}.
    Then we have the estimate 
    \begin{equation}
        \label{EQUATION:INCOMPLETE_STABILIZABILITY_ESTIMATE}
        e_{k}(t) + \int_{s}^{t} e_{k}(\tau) \dd \tau \leq C_{1} e_{k}(s) + C_{2} \int_{s}^{t} z^{3/2}(\tau) \dd \tau
    \end{equation}
    for $0 \leq s \leq t < T_{*}$ and $k \in \{0, 1, 2, 3\}$, where the constants $C_{k}$ depend neither on $t$ nor on $s$.
\end{cor}

In the linear or the Lipschitzian semilinear situation (cf.\ \cite{ElLaNi2002.2}),
one can restrict oneself to the case $k = 0$ and show that the latter term on the right-hand side of Equation \eqref{EQUATION:INCOMPLETE_STABILIZABILITY_ESTIMATE} vanishes.
In this situation, the squared norms $e_{0}(t)$ and $z_{0}(t)$ are equivalent. This renders the ``incomplete'' stabilizability inequality complete
and the exponential decay of $e_{0}(t) \sim z_{0}(t)$ follows similarly to the proof of Datko \& Pazy's theorem (cf.\ \cite[Remark 4.2]{LaPoSch2018}).
In contrast, in the quasilinear situation we are concerned with, such an argument is impossible at the basic energy level.
Hence, a genuinely quasilinear strategy needs to be developed.


\section{Uniform Stabilizability Inequality and Proof of Theorem \ref{thm:main}}
\label{SECTION:STABILIZABILITY}

The natural higher-order energies $e_{k}$ in the ``incomplete'' stabilizability estimate \eqref{EQUATION:INCOMPLETE_STABILIZABILITY_ESTIMATE}
in Corollary~\ref{observe} only contain higher-order temporal derivatives of the solution pair $(E, H)$ and, therefore, do not match the topology of the solution space $G^{3}$.
For the estimate \eqref{EQUATION:INCOMPLETE_STABILIZABILITY_ESTIMATE} to be complete,
the ``missing'' higher-order time-space derivatives need to be recovered so that both sides of the inequality can be formulated in terms of $z(t)$.
Trivially, the squared norms $e(t)$ and $z(t)$ are not equivalent in general.
At the same time, along classical solutions to the quasilinear Maxwell system \eqref{EQUATION:MAXWELL}, 
we will be able to bound $z(t)$ by a multiple of $e(t)$ plus a quadratic term in $z(t)$.
This amounts to a careful higher-order regularity analysis of solutions.
Namely, in Section \ref{SECTION:REGULARITY_BOOST}, we prove the following regularity boost result.
\begin{prop}
    \label{PROPOSITION:REGULARITY_BOOST}
    Under conditions of Proposition \ref{prop:energy}, there exist constants $c_{5}$ and $c_{6}$ independent of $t$ such that
    \begin{equation}
        \label{EQUATION:ENERGY_BOOST_ESTIMATE}
        z(t) \leq c_{5} e(t) + c_{6} z^2(t) \qquad \text{ for } t \in [0, T_{\ast}).
    \end{equation}
\end{prop}
  A result reminiscent of Proposition \ref{PROPOSITION:REGULARITY_BOOST} 
for local-in-time solutions to Equation \eqref{EQUATION:MAXWELL} has recently been obtained in \cite{SchSp2}.
Here, our goal is to improve upon these results by showing the estimate holds uniform in time, i.e., the constants do not blow-up, provided the initial
data are sufficiently small. This will result in a uniform stabilizability inequality \eqref{EQUATION:STABILIZABILITY_EQUALITY} stated in Proposition \ref{prop:z} below.
Combining Corollary \ref{observe} and Proposition \ref{PROPOSITION:REGULARITY_BOOST}, we derive \eqref{EQUATION:STABILIZABILITY_EQUALITY}
in a fashion vaguely reminiscent of the proof of Proposition~4.1 of \cite{LaPoSch2018}.
\begin{prop}[Nonlinear stabilizability inequality]
    \label{prop:z}
    Suppose the conditions of Theorem \ref{thm:main} hold true and that $\delta\in(0,\delta_1]$ in \eqref{est:delta} with the number
    $\delta_1$ from Proposition~\ref{prop:lower}.
	For all initial data satisfying the smallness condition \eqref{est:data} with  $r \in \big(0, r(\delta)\big]$
    the associated classical solution $(E, H)$ fulfills the uniform stabilizability inequality
    \begin{equation}
        \label{EQUATION:STABILIZABILITY_EQUALITY}
        z(t) + \int_{s}^{t} z(\tau) \dd \tau \leq 
        C \Big(z(s) + z^2(t) + \int_{s}^{t} z^{2}(\tau) \dd \tau\Big) \qquad \text{ for } \qquad 0 \leq s < t < T_{*},
    \end{equation}
    where the complete (squared) norm $z$ is defined in \eqref{def:dez} and the constant $C$ is independent of $t$ and the radius $r \in (0, r(\delta)]$.
\end{prop}

\begin{proof}
    Utilizing Equation \eqref{EQUATION:INCOMPLETE_STABILIZABILITY_ESTIMATE} from Corollary \ref{observe} (summed up over $k \in \{0, 1, 2, 3\})$
    and Equation \eqref{EQUATION:ENERGY_BOOST_ESTIMATE} from Proposition \ref{PROPOSITION:REGULARITY_BOOST}, we estimate
    \begin{align*}
        c_{5}^{-1} z(t) - c_{5}^{-1} c_{6} z^{2}(t) 
        + \int_{s}^{t} \Big(c_{5}^{-1} z(\tau) - c_{5}^{-1} c_{6} z^{2}(\tau)\Big) \dd \tau
        &\leq e(t) + \int_{s}^{t} e(\tau) \dd \tau \\
        &\leq C_{1} e(s) + 4 C_{2} \int_{s}^{t} z^{3/2}(\tau) \dd \tau \\
        &\leq C_{1} z(s) + 4 C_{2} \int_{s}^{t} z^{3/2}(\tau) \dd \tau
    \end{align*}
    and, therefore,
    \begin{align}
        \label{EQUATION:STABILIZABILITY_EQUALITY_PRECURSOR}
        \begin{split}
        c_{5}^{-1} z(t) + c_{5}^{-1} \int_{s}^{t} z(\tau) \dd \tau 
        \leq C_{1} z(s) + c_{5}^{-1} c_{6} z^{2}(t) + \int_{s}^{t} \Big(4 C_{2} z^{3/2}(\tau) + c_{5}^{-1} c_{6} z^{2}(\tau)\Big) \dd \tau.
        \end{split}
    \end{align}
    Using Young's inequality to estimate $z^{3/2}(\tau) \leq \tfrac{1}{2} \big(z(\tau) + z^{2}(\tau)\big)$, 
    the estimate in (\ref{EQUATION:STABILIZABILITY_EQUALITY_PRECURSOR}) yields the desired inequality
    with an appropriate positive constant $C$.
\end{proof}

Theorem \ref{thm:main} can now easily be proved, cf.\ \cite{LaPoSch2018} or \cite{LPW}.

\begin{proof}[Proof of Theorem \ref{thm:main}]
 Fixing $\delta=\min\{\delta_1, 1/2C\}$, estimate \eqref{EQUATION:STABILIZABILITY_EQUALITY}  furnishes the enhanced linear stabilizability inequality
    \begin{equation}
        \label{EQUATION:LINEAR_STABILIZABILITY_EQUALITY}
        z(t) + \int_{s}^{t} z(\tau) \dd \tau \leq 2 C z(s) \qquad \text{ for } 0 \le s \le t <T_*.
    \end{equation}
    If we have here $T_*=\infty$, the result follows as in the linear case, see \cite[Remark 4.2]{LaPoSch2018} or the proofs of \cite[Corollary 5.7]{LaPoWa2018} 
    or \cite[Theorem A.1]{AnPo2018}.
    
 Suppose that   $T_*<\infty$. Equation \eqref{eq:contr} then  yields $z(T_*)=\delta^2$.
 On the other hand,   \eqref{EQUATION:LINEAR_STABILIZABILITY_EQUALITY} shows that $z(t)\le 2Cz(0)$
 for $0\le t< T_*$ and initial data with $\|(E_0,H_0)\|_{\cH^3}^2 \le r^2$
for all $r\in (0, r(\delta)]$. (The number  $r(\delta)>0$ was introduced before \eqref{def:T*}.)
Formulas \eqref{eq:maxwell2} and \eqref{est:z} then imply
\[z(t) \le sCz(0)\le 2c_0C\, \|(E_0,H_0)\|_{\cH^3}^2\le 2  c_0 Cr^2\]
for a constant $c_0>0$. We now fix the radius
\begin{equation}
    \label{def:r} 
    r :=\min \Big\{r(\delta), \frac{\delta}{\sqrt{4c_0C}}\Big\}.
\end{equation}
As a result, the quantity $z(t)$ is bounded by $\delta^2/2$ for $t<T_*$ and by continuity also for $t=T_*$.
This fact contradicts  $z(T_*)=\delta^2$, and hence $T_*=\infty.$
\end{proof}

\section{Auxiliary Results}
\label{SECTION_AUXILIARY_RESULTS}

\subsection{$\pmb\curl$-$\pmb\div$-estimates}
The following lemma is a variant  of \cite[Lemma 4.5.5]{CoDaNi2010} for tensors $\lambda$.
In a certain sense, it ensures  optimal regularity for the ``elliptic part''  of 
Equation \eqref{EQUATION:LINEAR_INHOMOGENEOUS_MAXWELL_SYSTEM1}. 
\begin{lem}
    \label{prop:div-curl}
    Let $\Omega$ be simply connected. Assume that  $u, v \in \cH(\curl)$ and $\alpha, \beta \in W^{1, \infty}(\Omega, \mathbb{R}^{3 \times 3}_{\mathrm{sym}})$
    satisfy $\alpha, \beta \geq \eta > 0$, $\div(\alpha u), \div(\beta v) \in L^{2}(\Omega)$ and 
    $v \times \nu + \lambda (u \times \nu) \times \nu =: h \in \cH^{1/2}(\Gamma, \mathbb{R}^{3})$.
    Then $u$ and $ v$ belong to  $\cH^1(\Omega, \mathbb{R}^{3})$ and
    \begin{align*}
        \|u\|_{\cH^{1}} + \|v\|_{\cH^{1}} &\leq c\big(\|\curl u\|_{L^{2}} + \|\curl v\|_{L^{2}}
        + \|\div(\alpha u)\|_{L^{2}} + \|\div(\beta v)\|_{L^{2}} + \|h\|_{\cH^{1/2}(\Gamma)}\big).
    \end{align*}
\end{lem}

\begin{proof}
    Proposition~IX.1.3 and Remark~IX.1.4 of \cite{DL} yield a function $w \in \cH^{1}(\Omega, \mathbb{R}^{3})$ with  $\curl u = \curl w$, 
    $\div w = 0$, and $\nu \cdot w = 0$ on $\Gamma$.     Moreover, $\|w\|_{\cH^{1}} \leq c \|u\|_{L^{2}}$.
    Further, $u- w = \nabla \varphi$ for some $\varphi \in \cH^{1}(\Omega)$ by \cite[Proposition IX.1.2]{DL}.
    Hence,
    \begin{equation*}
        \div(\alpha \nabla \varphi) = \div(\alpha u) - \div(\alpha w) \in L^{2}(\Omega).
    \end{equation*}
    As in Definition~2.2 of \cite{Ce} we let  $\nabla_{\Gamma} \varphi = \nu\times (\nabla \ph\times \nu)$
    be the tangential trace of $\nabla\varphi$, and we define $\div_\Gamma$ by duality.  We can write  $u\times \nu =Bu$ with
    \[ B:=\begin{pmatrix} 0 &\nu_3 &-\nu_2\\ -\nu_3 & 0 & \nu_1\\ \nu_2 & -\nu_1& 0 \end{pmatrix}.\]
 The matrix $\tilde{\lambda}:=B^\top \lambda B$ is again symmetric, positive definite, and Lipschitzian in $x$. Observe that
 \[ \tilde{\lambda} (\nu\times (\nabla \ph\times \nu)) = B^\top \lambda (\nabla \ph\times \nu)=  -( \lambda (\nabla \ph\times \nu))\times \nu.\]
 Invoking \cite[Theorem 2.4]{Ce} and the boundary condition, we thus deduce
    \begin{equation*}
        \tilde\lambda \nabla_{\Gamma} \varphi = (\lambda (w \times \nu)) \times \nu - (\lambda (u \times \nu)) \times \nu
        = (\lambda (w \times \nu)) \times \nu + v \times \nu - h=:\psi\in  \cH^{-1/2}(\div_\Gamma),
    \end{equation*}
where we also use the fact $v\in \cH(\curl)$ and Theorem~2.4 in \cite{Ce}.
Hence, $\div_\Gamma (\tilde\lambda \nabla_{\Gamma} \varphi)\in H^{-1/2}(\Gamma)$
and the standard elliptic theory implies that $\tr \varphi$  belongs to $H^{3/2}(\Gamma)$. Again, by elliptic results, it follows that
$\varphi$ is contained in  $\cH^{2}(\Omega)$ and satisfies
    \begin{align*}
\|\varphi\|_{\cH^{2}(\Omega)} &\leq c\big(\|\div(\alpha u)\|_{L^{2}(\Omega)} + \|\div(\alpha w)\|_{L^{2}(\Omega)} + \|\div_\Gamma \psi\|_{\cH^{-1/2}(\Gamma)}\big) \\
        &\leq c\big(\|\div(\alpha u)\|_{L^{2}(\Omega)} + \|w\|_{\cH^{1}(\Omega)} + \|h\|_{\cH^{1/2}(\Omega)} + \|\curl v\|_{L^{2}(\Omega)}\big) \\
        &\leq c\big(\|\div(\alpha u)\|_{L^{2}(\Omega)} + \|\curl u\|_{L^{2}(\Omega)} + \|h\|_{\cH^{1/2}(\Omega)}+ \|\curl v\|_{L^{2}(\Omega)}\big).
    \end{align*}
We have thus shown the desired estimate for $u$. The inequality for $v$ is proved similarly, but here one can directly use that 
$\nu\times (v \times \nu)$ is bounded in $\cH^{1/2}(\Gamma)$ by $c\|\curl v\|_{L^{2}(\Omega)}$.
\end{proof}

\subsection{Tangential and normal derivatives}
\label{SUBSECTION:MICROLOCAL}

For a sufficiently small $a > 0$, consider the ``boundary colar'' $\Gamma_{a} := \big\{x \in \ol{\Omega} \,|\, \operatorname{dist}(x, \Gamma) < a\big\}$ of $\Gamma$
 with $\operatorname{dist}(x, \Gamma)  := \max_{y \in \Gamma} |x - y|$. Fix tangential vectors
 $\big\{\tau^{1}(x), \tau^{2}(x)\big\}$ spanning the tangential plane at each $x \in \Gamma$.
Since $\Gamma \in C^{5}$, the vector bundle $\big\{\tau^{1}(x), \tau^{2}(x), \nu(x)\big\}$ can be extended in $C^4$ 
from $\Gamma$ to $\Gamma_{a}$ such that $\big\{\tau^{1}(x), \tau^{2}(x), \nu(x)\big\}$ form an orthonormal basis of $\mathbb{R}^{3}$ for each $x\in \Gamma_a$.

Similar to \cite[Section 5.2]{LaPoSch2018}, we introduce the following differential operators.
For $\xi, \zeta \in \{\tau^{1}, \tau^{2}, \nu\}$, $u \in \mathbb{R}^{3}$ and $a \in \mathbb{R}^{3 \times 3}$, let
\begin{equation*}
    \partial_{\xi} = \sum_{j = 1}^{3} \xi_{j} \partial_{j}, \qquad
    u_{\xi} = u \cdot \xi, \qquad u^{\xi} = u_{\xi} \xi, \qquad
    u^{\tau} = \sum_{i = 1}^{2} u_{\tau^{i}} \tau^{i}, \qquad 
    a_{\xi \zeta} = \xi^{\top} a \zeta \qquad \text{ in } \Gamma_{a}.
\end{equation*}
With this notation, for a smooth function $\varphi$, we can write
\begin{equation*}
    \nabla \varphi = \sum_{\xi \in \{\tau^{1}, \tau^{2}, \nu\}} \xi \big(\xi \cdot \nabla \varphi\big)
    = \sum_{\xi \in \{\tau^{1}, \tau^{2}, \nu\}} \xi \partial_{\xi} \varphi \qquad \text{ and } \qquad
    \partial_{j} \varphi = \sum_{\xi \in \{\tau^{1}, \tau^{2}, \nu\}} \xi_{j} \partial_{\xi} \varphi \qquad \text{ in } \Gamma_{a}.
\end{equation*}
Therefore, observing the identity
\begin{equation*}
    \curl u = \partial_{1} \big(0, -u_{3}, u_{2}\big)^{\top} + \partial_{2} \big(u_{3}, 0, -u_{1}\big)^{\top} + \partial_{3} \big(-u_{2}, u_{1}, 0\big)^{\top},
\end{equation*}
the $\curl$-operator can be expressed as
\begin{equation*}
    \curl u = \sum_{j = 1}^{3} J_{j} \partial_{j} = \sum_{j = 1}^{3} \sum_{\xi \in \{\tau^{1}, \tau^{2}, \nu\}} J_{j} \xi_{j} \partial_{\xi} 
    = \sum_{\xi \in \{\tau^{1}, \tau^{2}, \nu\}} J(\xi) \partial_{\xi},
\end{equation*}
where
\begin{equation*}
    J_{1} =
    \begin{pmatrix}
        0 & 0 &  0 \\
        0 & 0 & -1 \\
        0 & 1 & 0
    \end{pmatrix}, \quad
    J_{2} =
    \begin{pmatrix}
         0 & 0 & 1 \\
         0 & 0 & 0 \\
        -1 & 0 & 0
    \end{pmatrix}, \quad
    J_{3} =
    \begin{pmatrix}
        0 & -1 & 0 \\
        1 &  0 & 0 \\
        0 &  0 & 0
    \end{pmatrix} \quad \text{ and } \quad
    J(\xi) = \sum_{j = 1}^{3} \xi_{j} J_{j}.
\end{equation*}
(The matrix $-J(\nu)$ was called $B$ in the proof  of Lemma~\ref{prop:div-curl}.) Note that 
\begin{equation*}
    \operatorname{ker}\big(J(\nu)\big) = \operatorname{span}\{\nu\} \qquad \text{ and } \qquad
    \big(\operatorname{ker}\big(J(\nu)\big)\big)^{\top} = \operatorname{span}\{\tau^{1}, \tau^{2}\big\} \qquad \text{ in } \Gamma_{a}.
\end{equation*}
In particular, factoring out the null space, we can write $J(\nu) u = J(\nu) u^{\tau}$.
Moreover, the restriction of $J(\nu)$ onto the span of $\{\tau^{1}, \tau^{2}\}$ is invertible with an inverse $R(\nu)$.

\subsection{Reconstructing the normal derivatives}

Typically, when showing the ``full'' Sobolev spatial regularity  for elliptic boundary-value problems,
the regularity of the boundary symbol proves very beneficial. Unfortunately, this property is violated by the Maxwell system.
Indeed, the associated boundary-value problem is characteristic since the boundary symbol $J(\nu)$ has a nontrivial kernel $\operatorname{span}\{\nu\}$ and, thus, is singular.
Hence, additional effort is needed to reconstruct the regularity in the normal component.
Following \cite[Section 5.2]{LaPoSch2018}, we employ the so-called ``$\curl$-$\div$-strategy'' which roots in the observation
that the $\curl$-operator stores the information about the normal derivative of the tangential components,
whereas the solenoidality condition furnishes an estimate for the normal derivatives of the normal component.
We briefly protocol this procedure below. By Subsection~\ref{SUBSECTION:MICROLOCAL}, we can write
\begin{equation*}
    \curl u = J(\nu) (\partial_{\nu} u)^{\tau}  + J(\tau^{1}) \partial_{\tau^{1}} u  + J(\tau^{2}) \partial_{\tau^{2}} u.
\end{equation*}
Using the invertibility of $J(\nu)|_{\{\tau^{1}, \tau^{2}\}}$, we get
\begin{equation}
    \label{eq:curl1}
    \begin{split}
        \partial_\nu u^{\tau} &= \sum_{i = 1}^{2} (\partial_\nu \tau^{i} \, u_{\tau^{i}} + \tau^{i} \partial_{\nu} \tau^{i} \cdot u)
        + R(\nu) \Big(\curl u - \sum_{i = 1}^{2} J(\tau^{i}) \partial_{\tau^{i}} u \Big) \\
        &= R(\nu) \Big(\curl u - \sum_{i = 1}^{2} J(\tau^{i}) \partial_{\tau^{i}} u \Big) + \mathrm{l.o.t.}(u),
    \end{split}
\end{equation}
where the lower-order terms are a zeroth-order operator determined by $\Omega$.

Proceeding to the divergence, for a smooth matrix-valued $\alpha$, we obtain
\begin{align*}
    \div(\alpha u) &= \sum_{\xi, \zeta} \partial_{\xi} \big(\alpha_{\xi \zeta} u_{\zeta}\big) + \div(\xi) \xi^{\top} \alpha u \\
    &= \sum_{\xi, \zeta} (a_{\xi \zeta} \partial_{\xi} u_{\zeta} + \partial_{\xi} \alpha_{\xi \zeta} u_\zeta) + \sum_{\xi} \div(\xi) \xi^{\top} \alpha u
\end{align*}
with $\alpha_{\xi \zeta} = \xi^{\top} \alpha \zeta$. Solving this equation for $\partial_{\nu} u_{\nu}$ yields
\begin{equation}
    \label{eq:div-nu}
    \alpha_{\nu \nu} \partial_{\nu} u_{\nu}  = \div(\alpha u) - \sum_{(\xi,\zeta) \neq(\nu, \nu)} \alpha_{\xi \zeta} \partial_{\xi} u_{\zeta}
     - \sum_{\xi, \zeta} \partial_{\xi} a_{\xi \zeta} u_{\zeta} - \sum_{\xi} \div(\xi) \xi^{\top} \alpha u.
\end{equation}
If $\alpha=\alpha^\top \ge \eta I$, we have $\alpha_{\nu \nu}\ge \eta$ and thus the normal derivative $\partial_{\nu} u_{\nu}$ of the normal component 
 of $u$ is representeted through $\div(\alpha u)$, the normal derivatives of the tangential components of $u$, the tangential derivatives of $u$, 
 and  lower-order terms depending on $\Omega$ and $\alpha$ along with its derivatives.

In contrast to \cite{LaPoSch2018}, where the solenoidality condition was violated for the electic field 
due to the presence of an internally distributed electrical resistance term,
both fields in \eqref{EQUATION:MAXWELL} satisfy appropriate solenoidality conditions
and can be handled via Equation \eqref{eq:div-nu}.


\section{Regularity Boost: Proof of Proposition \ref{PROPOSITION:REGULARITY_BOOST}}
\label{SECTION:REGULARITY_BOOST}

\begin{proof}
We first outline the argument.  In Step 1), the desired estimates of  $\partial_{t}^{k} E$ and $\partial_{t}^{k} H$ easily follow from 
our ``elliptic'' $\curl$-$\div$-estimate in Lemma~\ref{prop:div-curl} and the system \eqref{eq:maxwell1}. 
Similarly, one can also bound the higher-order terms away from the boundary in the second step.
Near the boundary, one has to distinguish between the tangential and the normal regularity. The former can again be treated by means of  
Lemma~\ref{prop:div-curl} since tangential derivatives preserve the boundary condition up to lower-order terms. The normal regularity 
has to be recovered from the Maxwell equations and the divergence conditions as explained in the previous section. Iteratively, this is repeated in Step~4) 
for $(E,H)$ in $\cH^2$, in Step~5) for $\partial_t(E,H)$ in $\cH^2$ and in step~6) for $(E,H)$  in $\cH^3$.

 \smallskip

\textit{1) $\cH^{1}$-estimates for $\partial_{t}^{k} E$ and $\partial_{t}^{k} H$.}
   Let $k\in\{0,1,2\}$.
    We first control the $L^2$-norms of the curl and divergence of these fields via the time-differentiated Maxwell system \eqref{eq:maxwell1}
    and the solenoidality conditions  in \eqref{eq:div0} and \eqref{eq:div}. Also invoking 
the \emph{a priori} estimates \eqref{est:z}, we arrive at
    \begin{align*}
        \big\|\curl \partial_{t}^{k} E(t)\big\|_{L^2(\Omega)} &\leq c e_{k+1}^{1/2}(t) + cz(t)\delta_{k = 2}, &
        \big\|\curl \partial_{t}^{k} H(t)\big\|_{L^2(\Omega)} &\leq c e_{k+1}^{1/2}(t) + cz(t)\delta_{k = 2}, \\
        \big\|\!\div(\wh\ep_{k} (t)\partial_{t}^{k} E(t))\big\|_{L^{2}(\Omega)} &\leq cz(t)\delta_{k = 2}, & 
        \big\|\!\div(\wh\mu_{k}(t) \partial_{t}^{k} H(t))\big\|_{L^{2}(\Omega)} &\leq cz(t)\delta_{k = 2}
        %
    \end{align*}
    with the Kronecker delta $\delta_{k = 2} = 1$ for $k = 2$ and $\delta_{k = 2}=0$ for $k \in \{0, 1\}$.
    Lemma~\ref{prop:div-curl} thus yields
    \begin{equation}
        \label{est:EH-H1}
        \begin{split}
            \big\|\partial_{t}^{k} E(t)\big\|_{\cH^{1}(\Omega)}^{2} \leq c e_{k+1}(t) + c z^{2}(t)\delta_{k = 2}, \qquad
            \big\|\partial_{t}^{k} H(t)\big\|_{\cH^{1}(\Omega)}^{2} \leq c e_{k+1}(t) + c z^{2}(t)\delta_{k = 2}.
        \end{split}
    \end{equation}
     (Later on we often omit the argument $t$.)
    Note that the right-hand sides $\partial_{t}f_{k}$ and $\partial_{t}g_{k}$ in Equation \eqref{eq:maxwell1} are (locally) quadratic with respect to $(E, H)$
    and can thus be dominated by $z$ according to Equation \eqref{est:z}.
    This results in a superlinear remainder term of order $2$ in $z$ and will prove to be crucial for our reasoning.
    
    \smallskip
    
    \textit{2) Interior spatio-temporal estimates for $E$ and $H$.} Let $k \in \{0, 1\}$.
    For a fixed sufficiently small  $a > 0$, we have the boundary ``collar'' $\Gamma_{a} = \big\{x \in \ol{\Omega} \,|\, \operatorname{dist}(x, \Gamma) < a\big\}$
    from Subsection~\ref{SUBSECTION:MICROLOCAL}.
    Consider a cut-off function $\chi \in C_{c}^{\infty}\big(\ol{\Omega}, [0, 1]\big)$ with $\operatorname{supp}(\chi) \subset \Gamma_{a}$
    and $\chi =1$ on  $\Gamma_{a/2}$.
    Then  $1 - \chi$ is smooth with $\operatorname{supp}(1 - \chi) \subset \Omega \backslash \Gamma_{a/2}$ and, hence, vanishes near the boundary $\Gamma$.
    Moreover, the maps $(1 - \chi) E$ and $(1 - \chi) H$ are as regular as $E$ and $H$ and satisfy the Maxwell system \eqref{eq:maxwell1} 
    with the new right-hand sides
    \begin{align}
        \label{EQUATION:NEW_RHS_INTERIOR}
        \wh{f}_{k} =  \nabla \chi \times \partial_{t}^{k} H, \qquad
        \wh{g}_{k} =  - \nabla \chi \times \partial_{t}^{k} E
    \end{align}
    in place of $-\partial_{t} f_{k}=0$ respectively  $-\partial_{t} g_{k}=0$.

    Take  a multi-index $\alpha \in \NN_{0}^{3}$ with $1\le |\alpha| \leq 2 - k$. Using Equation \eqref{eq:maxwell1}, we calculate
    \begin{align*}
        \curl\big(\partial_{t}^{k} \partial_{x}^{\alpha} \big((1 - \chi) H\big)\big) =
        \partial_{t} \partial_{x}^{\alpha} \big(\wh{\ep}_{k} (1 - \chi) \partial_{t}^{k} E\big) - \partial_{x}^{\alpha} \big(\nabla \chi \times \partial_{t}^{k}H\big).
    \end{align*}
    Employing also the estimates of $\wh\ep_k$ in \eqref{est:z}, we derive
    \begin{align*}
        \big\|\curl\big(\partial_{t}^{k} \partial_{x}^{\alpha} \big((1 - \chi) H\big)\|_{L^{2}(\Omega)} \leq
        c \max_{0 \leq j \leq 3-|\alpha|} \|\partial_{t}^{j} E\|_{\cH^{|\alpha|}(\Omega)} + c\|\partial_{t}^{k} H\|_{H^{|\alpha|}(\Omega)}.
    \end{align*}
(One also invokes the Sobolev embeddings $\cH^1(\Omega) \hookrightarrow L^6(\Omega)$ and $\cH^2(\Omega) \hookrightarrow L^\infty(\Omega)$ to
bound $\partial_{t}^{j} E$.)  Note that $\partial_{t}^{k} \partial_{x}^{\alpha} \big((1 - \chi) E)$ and  $\partial_{t}^{k} \partial_{x}^{\alpha} \big((1 - \chi) H)$
 trivially  satisfy the boundary condition in  \eqref{eq:maxwell1}. We  recall the divergence properties
    \begin{align*}
        \div\big(\wh\mu_{0} H\big) = 0 \qquad \text{and} \qquad \div\big(\wh\mu_{1} \partial_t H\big) = 0,
    \end{align*}
    from Equations \eqref{eq:div0} and \eqref{eq:div}. Applying the product rule, it follows
    \begin{align}
       \label{EQUATION:INTERIOR_DIVERGENCE_ESTIMATE}
        \div\big(\wh{\mu}_{k} \partial_{t}^{k} \partial_{x}^{\alpha} \big((1 - \chi) H\big)\big) 
        &= \sum_{0 \leq \beta \le \alpha} \binom{\alpha}{ \beta} 
        \partial_{x}^{\alpha - \beta} \big((\nabla \chi)^{\top} \wh{\mu}_{k}\big) \partial_{t}^{k} \partial_{x}^{\beta} H \\
        \notag 
        &\qquad- \sum_{0 \leq \beta < \alpha} \binom{\alpha}{ \beta} 
          \div\big(\partial_{x}^{\alpha - \beta} \wh{\mu}_{k} \partial_{t}^{k} \partial_{x}^{\beta} \big((1 - \chi) H\big)\big) .
    \end{align}
    The inequalities in \eqref{est:z} yield
    \begin{align}
        \begin{split}
            \big\|\div\big(\wh{\mu}_{k} \partial_{t}^{k} \partial_{x}^{\alpha} \big((1 - \chi) H\big)\big)\big\|_{L^{2}(\Omega)} &\leq
            c \|\partial_{t}^{k} H\|_{\cH^{|\alpha|}(\Omega)}, \\ 
            \big\|\div\big(\wh{\ep}_{k} \partial_{t}^{k} \partial_{x}^{\alpha} \big((1 - \chi) E\big)\big)\big\|_{L^{2}(\Omega)} &\leq
            c \|\partial_{t}^{k} E\|_{\cH^{|\alpha|}(\Omega)}, 
        \end{split}
    \end{align}
    where the latter estimate is shown analogously.
    
    First, let $|\alpha|=1$.  Lemma~\ref{prop:div-curl} and inequality  \eqref{est:EH-H1} then  imply the estimates
    \begin{equation}
        \label{EQUATION:INTERIOR_SPATIO_TEMPORAL_ESTIMATE}
        \big\|\partial_{t}^{k} \big((1 - \chi) E(t)\big)\|_{\cH^{2}(\Omega)}^2, \ \big\|\partial_{t}^{k} \big((1 - \chi) H(t)\big)\|_{\cH^{2}(\Omega)}^2
        \leq c e(t) + cz^2(t).
    \end{equation}
   For $|\alpha|=2$, it similarly follows
 \begin{equation}
        \label{EQUATION:INTERIOR_SPATIO_TEMPORAL_ESTIMATE3}
        \| (1 - \chi) (E(t),H(t))\|_{\cH^{3}(\Omega)}^2
        \leq c \max_{j\in\{0,1\}}\big\{\|\partial_t^j E(t)\big)\|_{\cH^{2}(\Omega)}^2, \|\partial_t^j H(t)\big)\|_{\cH^{2}(\Omega)}^2\big\}.
    \end{equation}
    
    \smallskip
    
    \textit{3) Preparations for boundary collar estimates for $E$ and $H$.} Let $k\in\{0,1\}$ and $\tau\in\{\tau^1,\tau^2\}$.
   We next treat  $\chi E$ and $\chi H$. Similar as $((1 - \chi) E,(1 - \chi) H)$, the fields $\chi E$ and $\chi H$ solve 
   the Maxwell equations \eqref{eq:maxwell1} with 
    \begin{align*}
        \check{f}_{k} = - \nabla \chi \times \partial_{t}^{k} H, \qquad
        \check{g}_{k} =  \nabla \chi \times \partial_{t}^{k} E
    \end{align*}
    in place of $-\partial_{t} f_{k}=0$ respectively $-\partial_{t} g_{k}=0$. 
    An analogue of divergence equation \eqref{EQUATION:INTERIOR_DIVERGENCE_ESTIMATE} remains true.
    While the boundary condition is preserved (since $\chi = 1$ near $\Gamma$),
    the latter is only invariant under the tangential derivatives $\partial_{\tau} = \tau \cdot \nabla$ and not for the normal one.
    Taking the tangential derivative of  the boundary condition in Equation \eqref{eq:maxwell1}, we can write
    \begin{align}
        \label{EQUATION:BC_TANGENTIALLY_DIFFERENTIATED}
            \partial_{\tau} \partial_{t}^{k} H \times \nu + \big(\lambda \big(\partial_{\tau} \partial_{t}^{k} E \times \nu)\big) \times \nu
            &= -\partial_{t}^{k} H \times \partial_{\tau} \nu - \big(\partial_{\tau} \lambda(\partial_{t}^{k} E \times \nu)\big) \times \nu \\
& \qquad- \big(\lambda (\partial_{t}^{k} E \times \partial_{\tau} \nu)\big) \times \nu -\big(\lambda (\partial_{t}^{k} E \times \nu)\big) \times \partial_{\tau} \nu\notag\\
            &=: h_{k}.\notag
    \end{align}

\smallskip
    
    \textit{4) $\cH^{2}$-estimate for $E$ and $H$.}
    An application of the tangential derivative to the first two equations in \eqref{EQUATION:MAXWELL} produces
    \begin{equation}
        \label{EQUATION:CURL_ESTIMATE_COLLAR}
        \begin{split}
            \partial_{t} \big(\wh{\ep}_{0} \partial_{\tau} (\chi E)\big) &= \curl \partial_{\tau} (\chi H) +
            [\partial_{\tau}, \curl] (\chi H) - \partial_{\tau} (\nabla \chi \times H) -  \partial_{t} \big(\partial_{\tau} \wh{\ep}_{0} (\chi E)\big), \\
            \partial_{t} \big(\wh{\mu}_{0} \partial_{\tau} (\chi H)\big) &= -\curl \partial_{\tau} (\chi H) -
            [\partial_{\tau}, \curl] (\chi E) + \partial_{\tau} (\nabla \chi \times E) -\partial_{t} \big( \partial_{\tau}\wh{\mu}_{0}(\chi H)\big)
        \end{split}
    \end{equation}
    (cf.\ Equation \eqref{EQUATION:NEW_RHS_INTERIOR}), where $[A, B] = AB - BA$ stands for the commutator.
    Thus, in view of Equation \eqref{est:EH-H1}, we can estimate
    \begin{equation*}
        \big\|\curl \partial_{\tau} \big(\chi H(t)\big)\big\|_{L^{2}(\Omega)}^2 \leq
        c\big(\|H(t)\|_{\cH^{1}(\Omega)}^2 + \|E(t)\|_{\cH^{1}(\Omega)}^2 + \|\partial_{t} E(t)\|_{\cH^{1}(\Omega)}^2\big) \leq c e(t).
    \end{equation*}
    Similarly, observing
    \begin{equation}
        \label{EQUATION:DIV_IDENTITY_COLLAR}
        \div\big(\wh{\mu}_{0} \partial_{\tau} (\chi H)\big) = -[\partial_{\tau}, \div] (\wh{\mu}_{0} \chi H) 
            + \partial_{\tau} (\nabla \chi \cdot \wh{\mu}_{0} H)  -\div\big( \partial_{\tau} \wh{\mu}_{0}(\chi H)\big) ,
    \end{equation}
    we can bound the divergence by
    \begin{equation}
        \label{EQUATION:DIV_ESTIMATE_COLLAR}
        \big\|\div\big(\wh{\mu}_{0}(t) \partial_{\tau} (\chi H(t))\big)\big\|_{L^{2}(\Omega)}^2 \leq c \|H(t)\|_{\cH^{1}(\Omega)}^2 \leq c e(t).
    \end{equation}
   Equations \eqref{EQUATION:BC_TANGENTIALLY_DIFFERENTIATED} with $k=0$ and \eqref{est:EH-H1} yield
    \begin{align*}
        \big\|\partial_{\tau} (\chi H(t)) \times \nu + \lambda \big(\partial_{\tau} (\chi E(t)) \times \nu\big) \times \nu\big\|_{\cH^{1/2}(\Gamma)}^2
        \leq c \big(\|H(t)\|_{\cH^{1}(\Omega)}^2 + \|E(t)\|_{\cH^{1}(\Omega)}^2\big) \leq c e(t).
    \end{align*}
    Repeating the procedure for $E$ and invoking Lemma~\ref{prop:div-curl}, we arrive at
    \begin{equation}
        \label{EQUATION:TANGENTIAL_ESTIMATE}
        \big\|\partial_{\tau} \big(\chi E(t)\big)\big\|_{\cH^{1}(\Omega)}^2, \big\|\partial_{\tau} \big(\chi H(t)\big)\big\|_{\cH^{1}(\Omega)}^2 \leq c e(t).
    \end{equation}
    
    Turning to the normal derivatives, we apply $\partial_{j}$  to the first equation in \eqref{EQUATION:MAXWELL} and obtain
    \begin{equation*}
        \partial_{t} \big(\wh{\ep}_{0} \partial_{j} (\chi E)\big) = \curl \big(\partial_{j} (\chi H)\big) - \partial_{j} (\nabla \chi \times H)
                                         -  \partial_{t} \big(\partial_j \wh{\ep}_{0} (\chi E)\big)
    \end{equation*}
    for $j\in\{1,2,3\}$. Using \eqref{eq:curl1} with $u = \partial_{j} (\chi H)$ as well as inequalities
   \eqref{EQUATION:TANGENTIAL_ESTIMATE} and \eqref{est:EH-H1}, we compute
    \begin{equation}
        \label{EQUATION:NORMAL_ESTIMATE}
        \big\|\partial_{\nu} \big(\partial_{j} (\chi H(t))\big)_{\tau}\big\|_{L^{2}} ^2
        \leq c\big(\|\partial_{\tau} \partial_{j} (\chi H(t))\big\|_{L^{2}}^2 + \|H(t)\|_{\cH^{1}}^2 +\|E(t)\|_{\cH^{1}}^2 +\| \partial_t E(t)\|_{\cH^{1}}^2\big) 
        \leq c e(t).
    \end{equation}
    An analogous estimate follows for $E$. As in \eqref{EQUATION:DIV_IDENTITY_COLLAR}, we can write
    \begin{equation*}
        \div\big(\wh{\mu}_{0} \partial_{j} (\chi H)\big) = \partial_j\big(\nabla \chi \cdot \wh{\mu}_{0}H\big) - \div \big(\partial_{j} \chi  \wh{\mu}_{0} H \big).
    \end{equation*} 
    Therefore, employing \eqref{eq:div-nu}, \eqref{est:EH-H1} and \eqref{EQUATION:NORMAL_ESTIMATE}, we conclude
    \begin{align}
        \label{EQUATION:NORMAL_NORMAL_ESTIMATE}
            \big\|\partial_{\nu} \big(\partial_{j} (\chi H(t))\big)_{\nu}\big\|_{L^{2}(\Omega)}^2 &\leq
            c \big(\big\|\partial_{\nu} \big(\partial_{j} (\chi H(t))\big)_\tau\big\|_{L^{2}(\Omega)}^2 + \|H(t)\|_{\cH^{1}(\Omega)}^2\big) \le ce(t).
    \end{align}
    Repeating this argumentation for $H$, we establish 
    \[ \|\chi E(t)\|_{\cH^{2}(\Omega)}^2, \|\chi H(t)\|_{\cH^{2}(\Omega)}^2 \leq c e(t).\]
    This inequality and \eqref{EQUATION:INTERIOR_SPATIO_TEMPORAL_ESTIMATE} with $k=0$ imply the  $\cH^{2}$-estimate
    \begin{equation}
        \label{EQUATION:H2_ESTIMATE}
        \|E(t)\|_{\cH^{2}(\Omega)}^2, \|H(t)\|_{\cH^{2}(\Omega)}^2 \leq c e(t)+ cz^2(t).
    \end{equation}
    
    \smallskip

    \textit{5) $\cH^{2}$-estimate for $\partial_{t} E$ and $\partial_{t} H$.}
    We apply the $\partial_{\tau}$-operator to \eqref{eq:maxwell1} and  derive
    \begin{align}
        \label{EQUATION:NORMAL_AND_TIME_DERIVATITE_IDENTITY}
            \partial_{t} \big(\wh{\ep}_{1} \partial_{\tau} \partial_{t} (\chi E)\big) &=\curl \big(\partial_{\tau} \partial_{t} (\chi H)\big)
            + [\partial_{\tau}, \curl] \big(\partial_{t} (\chi H)\big) \notag\\
            &\qquad-\partial_{\tau} (\nabla \chi \times \partial_{t} H) - \partial_{t} \big(\partial_{\tau} \wh{\ep}_{1} \partial_{t} (\chi E)\big), \\
            \partial_{t} \big(\wh{\mu}_{1} \partial_{\tau} \partial_{t} (\chi H)\big) &= -\curl \big(\partial_{\tau} \partial_{t} (\chi E)\big)
            - [\partial_{\tau}, \curl] \big(\partial_{t} (\chi E)\big) \notag\\
            &\qquad+ \partial_{\tau} (\nabla \chi \times \partial_{t} E) - \partial_{t} \big(\partial_{\tau} \wh{\mu}_{1} \partial_{t} (\chi H)\big).\notag
    \end{align}
    Recycling the $\cH^{1}$-estimate \eqref{est:EH-H1}, it follows
    \begin{equation*}
        \big\|\curl \partial_{\tau} \partial_{t} (\chi E)\big\|_{L^{2}(\Omega)} \leq
        c\big(\|\partial_{t}^{2} H\|_{\cH^{1}(\Omega)} + \|\partial_{t} H\|_{\cH^{1}(\Omega)} + \|\partial_{t} E\|_{\cH^{1}(\Omega)}\big) 
        \leq c e^{1/2}(t) +cz(t).
    \end{equation*}
We take the tangential derivative of the identity
    \begin{equation*}
        \div\big(\wh{\ep}_{1} \partial_{t} (\chi E)\big) = -\nabla \chi \wh{\ep}_{1} \partial_{t} E
    \end{equation*}
    and conclude
    \begin{align}
        \label{EQUATION:NORMAL_AND_TIME_DERIVATITE_DIVERGENCE_IDENTITY}
        \div\big(\wh{\ep}_{1} \partial_{\tau} \partial_{t} (\chi E)\big) = -[\partial_{\tau}, \div] \big(\wh{\ep}_{1} \partial_{t} (\chi E)\big)
         -\div\big(\partial_{\tau} \wh{\ep}_{1} \partial_{t} (\chi E)\big) - \partial_{\tau} \big(\nabla \chi \wh{\ep}_{1} \partial_{t} E\big).
    \end{align}
    In view of \eqref{est:EH-H1}, this identity furnishes the inequality
    \begin{equation*}
        \big\|\div\big(\wh{\ep}_{1}(t) \partial_{\tau} \partial_{t} (\chi E(t))\big)\big\|_{L^{2}(\Omega)} \leq c \|\partial_{t} E(t)\|_{\cH^{1}(\Omega)}
        \leq c e^{1/2}(t).
    \end{equation*}
    Observe that $\partial_{\tau} \partial_{t} (\chi E)$ satisfies the boundary condition \eqref{EQUATION:BC_TANGENTIALLY_DIFFERENTIATED} with
    \begin{equation*}
        \|h_1(t)\|_{H^{1/2}(\Omega)} \leq c\big(\|\partial_{t} E(t)\|_{H^{1}(\Omega)} + \|\partial_{t} H(t)\|_{H^{1}(\Omega)}\big) \leq c e^{1/2}(t),
    \end{equation*}
    where we again employed the $\cH^1$-estimate \eqref{est:EH-H1}. Lemma~\ref{prop:div-curl} thus implies the tangential bound
    \begin{equation}
        \label{EQUATION:NORMAL_AND_TIME_DERIVATITE_COLLAR_ESTIMATE}
        \big\|\partial_{\tau} \partial_{t} (\chi E(t))\big\|_{\cH^{1}(\Omega)}^2, \ \big\|\partial_{\tau} \partial_{t} (\chi H(t))\big\|_{\cH^{1}(\Omega)}^2
        \leq c e(t).
    \end{equation}

    We next look at the normal derivative.  An application of the $\partial_{j}$-operator to \eqref{eq:maxwell1} yields
    \begin{align*}
        \partial_{t} \big(\wh{\ep}_{1} \partial_{j} \partial_{t} (\chi E)\big) = \curl \partial_{j} \partial_{t} (\chi H)
        - \partial_{j} (\nabla \chi \times \partial_{t} H) - \partial_{t} \big(\partial_{j} \wh{\ep}_{1} \partial_{t} (\chi E)\big),
    \end{align*}
   cf.\ \eqref{EQUATION:NORMAL_AND_TIME_DERIVATITE_IDENTITY}.
    By means of  \eqref{eq:curl1} with $u = \partial_{j} \partial_{t} (\chi H)$, \eqref{est:EH-H1} and 
    \eqref{EQUATION:NORMAL_AND_TIME_DERIVATITE_COLLAR_ESTIMATE}, we deduce
    \begin{align}
        \label{EQUATION:NORMAL_ESTIMATE_TIME_DERIVATIVE}
            \big\|\partial_{\nu} \big(\partial_{j} \partial_{t} (\chi H)\big)_{\tau}\big\|_{L^{2}(\Omega)} ^2
            &\leq c\big(\|\partial_{t} H\|_{\cH^{1}(\Omega)}^2 + \|\partial_{\tau} \partial_{t} (\chi H)\|_{\cH^{1}(\Omega)}^2 
            + \|\partial_{t} (\chi E)\|_{\cH^{1}(\Omega)}^2 + \|\partial_{t}^{2} (\chi E)\|_{\cH^{1}(\Omega)}^2\big)\notag\\
           & \leq c e(t) + cz^2(t).
    \end{align}
On the other hand, \eqref{eq:div} yields
    \begin{equation*}
        \div\big(\wh{\mu}_{1} \partial_{j} \partial_{t} (\chi H)\big) =
        \partial_{j} \big(\nabla \chi \cdot\wh{\mu}_{1} \partial_{t} H\big) - \div\big(\partial_{j} \wh{\mu}_{1} \partial_{t} (\chi H)\big).
    \end{equation*}
Exploiting \eqref{eq:div-nu}, \eqref{est:EH-H1}, 
    \eqref{EQUATION:NORMAL_AND_TIME_DERIVATITE_COLLAR_ESTIMATE} and \eqref{EQUATION:NORMAL_ESTIMATE_TIME_DERIVATIVE}, we get the inequality
    \begin{align}
        \label{EQUATION:NORMAL_NORMAL_TIME_SPACE_DERIVATIVE_ESTIMATE}
            \big\|\partial_{\nu} \big(\partial_{j} \partial_{t} (\chi H(t))\big)_{\nu}\big\|_{L^{2}(\Omega)}^2
            &\leq c\big(\|\partial_{\tau} \partial_{t} (\chi H(t))\|_{\cH^{1}(\Omega)}^2
            +\big\|\partial_{\nu} \big(\partial_{j} \partial_{t} (\chi H(t))\big)_{\tau}\big\|_{L^{2}(\Omega)}^2 + \|\partial_{t} H(t)\|_{\cH^{1}(\Omega)}^2\big)\notag \\
           &\le c e(t) + cz^2(t).
    \end{align}
    A similar estimate follows in an analogous fashion for $E$. The bounds \eqref{EQUATION:NORMAL_AND_TIME_DERIVATITE_COLLAR_ESTIMATE}, 
    \eqref{EQUATION:NORMAL_ESTIMATE_TIME_DERIVATIVE} and \eqref{EQUATION:NORMAL_NORMAL_TIME_SPACE_DERIVATIVE_ESTIMATE}
    yield the collar estimate
    \begin{equation*}
        \|\chi \partial_{t} E(t)\|_{\cH^{2}(\Omega)}^2, \ \|\chi \partial_{t} H(t)\|_{\cH^{2}(\Omega)}^2 \leq c e(t) + cz^2(t).
    \end{equation*}
Combining this with  \eqref{EQUATION:INTERIOR_SPATIO_TEMPORAL_ESTIMATE}, we arrive at the $\cH^2$--inequality for $\partial_t(E,H)$
     \begin{equation}
        \label{EQUATION:H_2_ESTIMATE_TIME_DERIVARITE_COLLAR}
        \|\partial_{t} E(t)\|_{\cH^{2}(\Omega)}^2, \ \| \partial_{t} H(t)\|_{\cH^{2}(\Omega)}^2 \leq c e(t) +cz^2(t).
    \end{equation}
    Hence, \eqref{EQUATION:INTERIOR_SPATIO_TEMPORAL_ESTIMATE3} leads to the interior bound
\begin{equation}
        \label{est:interior3}
        \| (1 - \chi) (E(t),H(t))\|_{\cH^{3}(\Omega)}^2 \leq c e(t) +cz^2(t).
    \end{equation}
     
     \smallskip

    \textit{6) $\cH^{3}$-estimate for $E$ and $ H$.}
    As in \eqref{EQUATION:CURL_ESTIMATE_COLLAR}, letting $\partial_{\tau}^{2} := \partial_{\tau_{i}} \partial_{\tau_{j}}$ for $i, j \in \{1, 2, 3\}$,
    we can write
    \begin{align*}
        \partial_{t} \big(\wh{\ep}_{0} \partial_{\tau}^{2} (\chi E)\big) &= \curl \partial_{\tau}^{2} (\chi H)
        + [\partial_{\tau}^{2}, \curl] (\chi H) - \partial_{\tau}^{2} (\nabla \chi \times H) 
         - 2\partial_{t} \big(\partial_{\tau} \wh{\ep}_{0} \partial_{\tau} (\chi E)\big)
        - \partial_{t} \big( \partial_{\tau}^2\wh{\ep}_{0} \chi E \big).
    \end{align*}
    Using the inequalities \eqref{EQUATION:H2_ESTIMATE} and \eqref{EQUATION:H_2_ESTIMATE_TIME_DERIVARITE_COLLAR}, we derive
    \begin{equation*}
        \big\|\curl \partial_{\tau}^{2} (\chi H(t))\big)\big\|_{L^{2}(\Omega)}^2 \leq c e(t)+cz^2(t).
    \end{equation*}
    As before, a similar estimate can be shown for $E$.
    In the same fashion as in \eqref{EQUATION:DIV_IDENTITY_COLLAR}, we compute
    \begin{align*}
            \div\big(\wh{\mu}_{0} \partial_{\tau}^{2} (\chi H)\big) &= -[\partial_{\tau}^{2}, \div] (\wh{\mu}_{0} \chi H)
                 + \partial_{\tau}^{2} (\nabla \chi \cdot \wh{\mu}_{0} H) 
            - 2\div\big(\partial_{\tau} \wh{\mu}_{0} \partial_{\tau} (\chi H)\big) - \div\big(\partial_{\tau}^2\wh{\mu}_{0} (\chi H )\big).
    \end{align*}
    Thus, on the strength of \eqref{EQUATION:H2_ESTIMATE}, we get
    \begin{equation*}
        \big\|\div\big(\wh{\mu}_{0}(t) \partial_{\tau}^{2} (\chi H(t))\big)\big\|_{L^{2}(\Omega)}^2 \leq ce(t)+cz^2(t).
    \end{equation*}
A modification of \eqref{EQUATION:BC_TANGENTIALLY_DIFFERENTIATED} further yields 
    \begin{align*}
        \big\|\partial_{\tau}^{2} (\chi H) \times \nu + \big(\lambda (\partial_{\tau}^{2} (\chi E) \times \nu)\big) \times \nu\big\|_{\cH^{1/2}(\Gamma)}^2
        \leq c\big(\|E(t)\|_{\cH^{2}(\Omega)}^2 + \|H(t)\|_{\cH^{2}(\Omega)}^2\big) \leq c e(t)+cz^2(t).
    \end{align*}
    In summary, Lemma~\ref{prop:div-curl} implies the tangential bound 
    \begin{align}
        \label{EQUATION:H_1_DOUBLE_TANGENTIAL_DERIVATIVE_ESTIMATE}
        \|\partial_{\tau}^{2} (\chi H(t))\|_{\cH^{1}(\Omega)}^2, \ \|\partial_{\tau}^{2} (\chi E(t))\|_{\cH^{1}(\Omega)}^2 \leq c e(t)+cz^2(t).
    \end{align}

    We next bound the normal derivative. In spirit of \eqref{EQUATION:CURL_ESTIMATE_COLLAR}, we obtain
    \begin{align*}
            \partial_{t}\big(\wh{\ep}_{0} \partial_{\tau} \partial_{j} (\chi E)\big) 
	&= \curl \big(\partial_{\tau} \partial_{j} (\chi H)\big) + [\partial_{\tau}, \curl] \partial_{j} (\chi H)  - \partial_{\tau} \partial_{j} (\nabla \chi \times H)\\
            &\qquad - \partial_{t} \big(\partial_{\tau} \wh{\ep}_{0} \partial_{j} (\chi E)\big)- \partial_{t} \big(\partial_{j} \wh{\ep}_{0} \partial_{\tau} (\chi E)\big)
           - \partial_{t} \big(\partial_{\tau}\partial_j \wh{\ep}_{0}  (\chi E)\big)
    \end{align*}
    for $j\in\{1,2,3\}.$
    A joint application of \eqref{eq:curl1}, \eqref{EQUATION:H2_ESTIMATE}, \eqref{EQUATION:H_2_ESTIMATE_TIME_DERIVARITE_COLLAR}
    and  \eqref{EQUATION:H_1_DOUBLE_TANGENTIAL_DERIVATIVE_ESTIMATE} furnishes the estimate
    \begin{equation}
        \label{EQUATION_NORMAL_TANGENTIAL_TANGENTIAL_SPATIAL_DERIVATIVE_COLLAR_ESTIMATE}
        \big\|\partial_{\nu} \big(\partial_{\tau} \partial_{j} H(t))_{\tau}\big\|_{L^{2}(\Omega)}^2 \leq c e(t)+cz^2(t).
    \end{equation}
    We can treat $E$ analogously. As in \eqref{EQUATION:DIV_IDENTITY_COLLAR}, we compute
    \begin{align*}
        \begin{split}
            \div\big(\wh{\mu}_{0} \partial_{\tau} \partial_{j} (\chi H)\big) 
            &= -[\partial_{\tau}, \div] \big(\wh{\mu}_{0} \partial_{j} (\chi H) + \partial_{j}\wh{\mu}_{0}(\chi H)\big) 
              - \div\big(\partial_{\tau} \wh{\mu}_{0} \partial_{j} (\chi H)\big) \\
            &\qquad  - \div\big(\partial_{j} \wh{\mu}_{0} \partial_{\tau} (\chi H)\big)- \div\big(\partial_\tau\partial_{j} \wh{\mu}_{0} (\chi H)\big)
            + \partial_{\tau} \partial_{j} \big(\nabla \chi \cdot \wh{\mu}_{0} H\big).
        \end{split}
    \end{align*}
    Using formula \eqref{eq:div-nu} and estimates \eqref{est:z},  \eqref{EQUATION:H2_ESTIMATE}, 
    \eqref{EQUATION:H_1_DOUBLE_TANGENTIAL_DERIVATIVE_ESTIMATE} and \eqref{EQUATION_NORMAL_TANGENTIAL_TANGENTIAL_SPATIAL_DERIVATIVE_COLLAR_ESTIMATE}, we estimate
    \begin{equation}
        \label{EQUATION_NORMAL_TANGENTIAL_NORMAL_SPATIAL_DERIVATIVE_COLLAR_ESTIMATE}
        \big\|\partial_{\nu} \big(\partial_{\tau} \partial_{j} (\chi H(t))\big)_{\nu}\big\|_{L^{2}(\Omega)}^2 \leq c e(t)+cz^2(t).
    \end{equation}
    The electric field $E$ is controlled similarly. 
    Combining \eqref{EQUATION_NORMAL_TANGENTIAL_TANGENTIAL_SPATIAL_DERIVATIVE_COLLAR_ESTIMATE} and \eqref{EQUATION_NORMAL_TANGENTIAL_NORMAL_SPATIAL_DERIVATIVE_COLLAR_ESTIMATE},
    we arrive at
    \begin{equation}
        \label{EQUATION_NORMAL_TANGENTIAL_SPATIAL_DERIVATIVE_COLLAR_ESTIMATE}
        \big\|\partial_{\nu} \partial_{\tau} \partial_{j} (\chi E(t))\big\|_{L^{2}(\Omega)}^2 +
        \big\|\partial_{\nu} \partial_{\tau} \partial_{j} (\chi H(t))\big\|_{L^{2}(\Omega)}^2 \leq c e(t) + cz^2(t).
    \end{equation}
    Replacing $\partial_{\tau}$ by $\partial_{\nu}$ in the calculations above, we treat 
     $\partial_{\nu}^{2} \partial_{j} (\chi E)$ and $\partial_{\nu}^{2} \partial_{j} (\chi H)$ in the same way.
    Together with \eqref{EQUATION:H_1_DOUBLE_TANGENTIAL_DERIVATIVE_ESTIMATE} and \eqref{EQUATION_NORMAL_TANGENTIAL_SPATIAL_DERIVATIVE_COLLAR_ESTIMATE},
 this leads to the interior bound
    \begin{equation*}
        \|\chi E(t)\|_{\cH^{3}(\Omega)}^2 + \|\chi H(t)\|_{\cH^{3}(\Omega)}^2 \leq c e(t) +z^2(t).
    \end{equation*}
This inequality and \eqref{est:interior3} imply the final $\cH^3$--estimate
    \begin{equation}\label{EQUATION:H_3_COLLAR_ESTIMATE}
        \| E(t)\|_{\cH^{3}(\Omega)}^2 + \| H(t)\|_{\cH^{3}(\Omega)}^2 \leq c e(t) +z^2(t).
    \end{equation}
    The claim now follows from inequalities \eqref{est:EH-H1}, 
    \eqref{EQUATION:H_2_ESTIMATE_TIME_DERIVARITE_COLLAR}, and \eqref{EQUATION:H_3_COLLAR_ESTIMATE}.
\end{proof}


\bibliographystyle{plain}
\bibliography{maxwell-decay-bnd}


\end{document}